\documentclass{amsart}

\usepackage{mathtools}
\usepackage{amssymb}
\usepackage{amsmath}
\usepackage{amsfonts}
\usepackage{geometry}
\usepackage{bbm}
\usepackage{hyperref}
\usepackage{tikz}
\usepackage{mathrsfs}
\usepackage{multirow}
\usetikzlibrary{matrix,arrows,decorations.pathmorphing}
\usepackage{verbatim }
\usepackage{mathtools}
\usepackage[author={Sebastian}]{pdfcomment}
\usepackage{todonotes}

\newcommand{\myquad}[1][1]{\hspace*{#1em}\ignorespaces}

\newcounter{cprop}[section]
\renewcommand{\thecprop}{\thesection.\arabic{cprop}}
\newtheorem{theorem}[cprop]{Theorem}

\theoremstyle{plain}

\newtheorem{corollary}[cprop]{Corollary}

\newtheorem{lemma}[cprop]{Lemma}
\newtheorem{proposition}[cprop]{Proposition}

\newtheorem{assumption}[cprop]{Assumption}
\newtheorem{question}[cprop]{Question}
\numberwithin{equation}{section}

\theoremstyle{definition}
\newtheorem{definition}[cprop]{Definition}
\newtheorem{example}[cprop]{Example}

\theoremstyle{remark}
\newtheorem{remark}[cprop]{Remark}

\renewcommand{\P}{\mathbb{P}}

\newcommand{\R}{\mathbb{R}}
\newcommand{\N}{\mathbb{N}}
\newcommand{\Z}{\mathbb{Z}}

\renewcommand{\d}{\mathrm{d}}

\newcommand{\vertiii}[1]{{\left\vert\kern-0.25ex\left\vert\kern-0.25ex\left\vert #1 
		\right\vert\kern-0.25ex\right\vert\kern-0.25ex\right\vert}}

\begin{document}
	\title[Invariant manifolds for RDEs]{Invariant manifolds and stability for rough differential equations
	}
	
	\author{M. Ghani Varzaneh}
	\address{Mazyar Ghani Varzaneh\\
		Fachbereich Mathematik und Statistik, Universität Konstanz, Konstanz, Germany}
	\email{mazyar.ghani-varzaneh@uni-konstanz.de}

	\author{S. Riedel}
	\address{Sebastian Riedel \\
		Fakult\"at f\"ur  Mathematik und Informatik, FernUniversit\"at in Hagen, Hagen, Germany}
	\email{sebastian.riedel@fernuni-hagen.de }
	

	\begin{abstract}
		We prove the existence of local stable, unstable, and center manifolds for stochastic semiflows induced by rough differential equations driven by rough paths valued stochastic processes around random fixed points of the equation. Examples include stochastic differential equations driven by a fractional Brownian motion with Hurst parameter $H > \frac{1}{4}$. In case the top Lyapunov exponent is negative, we derive almost sure exponential stability of the solution.
	\end{abstract}
	
	\maketitle
	\section*{Introduction}
	
	Rough paths theory is a solution theory for ordinary differential equations that is rich enough to handle equations that are driven by paths with an arbitrary low H\"older regularity \cite{LCL07, FV10, FH20}. In particular, it can be used to study equations driven by Brownian trajectories and thus opens the possibility to study stochastic differential equations (SDEs) completely pathwise. This clear separation between \emph{probabilistic} aspects of the driving process and the \emph{deterministic} analysis of the equation makes it possible to define solutions to SDEs that are driven by very general driving signals. In particular, in contrast to It\=o's stochastic calculus, rough paths theory allows to study SDEs driven by stochastic processes lacking the martingale property. A famous class of stochastic processes serving as possible driving signals for rough differential equations are Gaussian processes and, most prominently among them, fractional Brownian motions with a Hurst parameter $H > \frac{1}{4}$ \cite{CQ02,FV10-2, FGGR16}. \smallskip
	
	Using SDEs driven by more general processes than Brownian motion can be more realistic in modelling real-world phenomena, but their analysis is significantly more complicated. This is due to the fact that the solution process lacks two properties that are heavily used in classical stochastic analysis: the \emph{martingale-} and the \emph{Markov property}. The question of how results that are known for SDEs driven by a Brownian motion can (or cannot) be generalized to equations driven by a fractional Brownian motion (or even more general Gaussian processes) has attracted many researcher's interest and still is an important, yet challenging problem. \smallskip
	
	One aspect of great interest is the question of how to describe the \emph{long-time behaviour} of the solution to a rough differential equation (RDE) driven by signals more general than Brownian motion. Let us recall that in case of classical SDEs, two core concepts that are used frequently to describe the long-time behaviour are \emph{invariant measures for the Markov semigroup} and \emph{stochastic stability} (see e.g. \cite{Kha12}). However, both concepts are not easy to generalize to rough differential equations. Concerning invariant measures, the lack of the Markov property does not even tell us how an invariant measure should be defined, leaving alone the questions of existence, uniqueness and convergence towards it. In a series of papers, Hairer and coauthors proposed a solution to this problems by generalizing the theory of invariant measures to solutions of equations driven by a fractional Brownian motion \cite{Hai05, HO07, HP11, HP13}. Several researchers adopted his ideas and used it to study related questions within this framework, cf. e.g. \cite{CP11, CPT14, FP17, DPT19, PTV20}. Concerning stochastic stability, there exist only a few works that deal with this problem. For a fractional Brownian motion with Hurst parameter $H > \frac{1}{2}$, this question was studied in \cite{GANS18, DHC19, DH23}. For lower Hurst parameters, we are only aware of the two works \cite{GAS18} and \cite{Duc22} that study the case $H \in (\frac{1}{3},\frac{1}{2})$. \smallskip
	
	Although invariant measures and stochastic stability are two important concepts, much more can be said about the long-time behaviour of SDE solutions by using the concepts of \emph{random dynamical systems} (RDS) \cite{Arn98}. Indeed, RDS offer very fine tools (Lyapunov exponents, invariant manifolds, random attractors...) that allow for a detailed description of the behaviour of SDE solutions. One important applications of RDS is the study of \emph{stochastic bifurcation}, i.e. qualitative changes of the solution that appear when perturbing the coefficients of the equation. For It\=o stochastic differential equations, the use of RDS is well-established. More importantly, RDS are flexible enough to deal with RDEs driven by stochastic processes having stationary increments such as fractional Brownian motions \cite{BRS17}. This makes RDS a perfect tool for the analysis of the long-time behaviour of RDE solutions.  \smallskip
	
	The present article offers a way to study the long-time behaviour of nonlinear RDE solutions by establishing the existence \emph{local random invariant manifolds} around stationary points. An invariant manifold has the property that it is invariant under the solution flow of the RDE. In fact, our main theorems, cf. Theorem \ref{stable_manifold}, Theorem \ref{Local unstable manifolds}, and Theorem \ref{center_manifold}, formulate sufficient conditions for the existence of stable, unstable, and center manifolds around stationary points and describe their properties. It is well-known that stable manifolds are closely related to exponential stability of the solution flow. Indeed, as a by-product of our stable manifold theorem, we can deduce local exponential stability for RDE solutions provided the largest Lyapunov exponent is negative, cf. Corollary \ref{stability}. We discuss an explicit example for which this property holds in Example \ref{STABLE_EXAMPLE}. \smallskip
	
	In the following, we compare our main results to existing theorems in the literature. For It\=o-SDEs, stable and unstable manifolds were established in \cite{MS99} and center manifolds are the topic of \cite{Box89}. Center manifolds for rough differential equations were recently studied in \cite{NK21}. Let us highlight some key features of our main results and how they are related to the results cited above.
	\begin{itemize}
		\item This article is the first that proves stable and unstable manifold theorems for rough differential equations. In particular, we think that the stable manifold theorem is an important result since it implies almost sure exponential stability of the solution provided that the top Lyapunov exponent is negative (cf.  Section \ref{sec:stability} and the discussion below). Compared to \cite{MS99}, note that we do not need to assume that the flow generated by our equation goes backward in time, too, i.e. we can drop the assumption that the cocycle should be injective. Note, however, that we are still able to prove an unstable manifold theorem although we cannot just apply the stable manifold theorem to the time-inversed flow which is a common strategy in the case of an injective cocycle. 
		\item Throughout the paper, we assume that the driving rough paths are geometric $\gamma$-H\"older paths with $\gamma > \frac{1}{4}$. Consequently, our main results can be applied to RDEs driven by a fractional Brownian motions with Hurst parameter $H > \frac{1}{4}$. Note that the authors of \cite{NK21} only consider the case $H > \frac{1}{3}$. Therefore, our paper establishes the first time a center manifold theorem for an RDE driven by a fractional Brownian motion with Hurst parameter $H \in (\frac{1}{4},\frac{1}{3}]$. Working with rough paths of lower H\"older regularity is technically more involved since we have to consider third-order iterated integrals and controlled rough paths with second Gubinelli-derivatives. However, we think that our arguments will even work for RDEs driven by geometric H\"older rough paths having arbitrary low H\"older regularity, but since our main example is the fractional Brownian motion, we refrained working in this generality and to keep the calculations as simple as possible.
		\item We prove the existence of stable, unstable, and center manifolds around stationary points that are allowed to be random. In contrast, the center manifold theorem in \cite{NK21} only applies if the equation has $0$ as a deterministic fixed point (in fact, in \cite{NK21}, it is even assumed that the first derivative of the drift and first and second derivative of the diffusion vector field have $0$ as a fixed point). There are many equations that fail to have deterministic fixed points but admit random ones, cf. the discussion in \cite[pages 15 - 18]{MS99}. The reason why we can formulate a more general result here is that we use the \emph{Mulitplicative Ergodic Theorem} that ensures the existence of Lyapunov exponents in a very general framework.
		\item In \cite{NK21}, the drift parameter in the RDE is assumed to be a linear map plus a Lipschitz continuous nonlinearity. In particular, the drift is assumed to have a linear growth. In many applications, this assumption is too restrictive (for example, it does not allow to study the important case $V_0(z) = -z|z|^2 + z$). Our main results are formulated in a generality that allows drift vector fields with superlinear growth by imposing e.g. only one one-sided growth conditions as formulated in \cite{RS17}.
		\item The stability result that we discuss in Section \ref{sec:stability} differs from those we mentioned above in several regards. First, it is formulated in a generality that allows to apply it to RDEs driven by a fractional Brownian motion with Hurst parameter $H > \frac{1}{4}$. In particular, this is the first stability statement that holds in the regime $H \in (\frac{1}{4}, \frac{1}{3}]$. Second, our assumptions on the equation are less restrictive. For instance, we are still able to prove stability when the derivative of the diffusion part is not necessarily equal to zero at the equilibrium point, cf. \cite[Equation (22)]{GAS18}, but is allowed to fluctuate around it. This, in particular, is interesting for studying possible bifurcations. Third, even for the simple regime $H \in (\frac{1}{3},\frac{1}{2}]$, our proof is much briefer and, we think, more conceptual than those given in \cite{GAS18} and \cite{Duc22}. Fourth, our result can be used to prove stability around any stationary point if some estimation of the Lyapunov exponent is provided.
	\end{itemize}
	\smallskip
	The article is structured as follows. In Section~\ref{sec:RDE}, we first provide background on rough path theory and prove several auxiliary lemmas. We then establish crucial estimates. The main results of this section are formulated in Propositions~\ref{RS17}, \ref{DDSD}, and \ref{ZXZ}, which serve as key tools for proving the principal results of this paper. Section~\ref{manifolds} presents our main contributions, where we introduce random fixed points for cocycles (termed \emph{stationary trajectories}), around which the invariant manifolds are constructed. The main results of this section are Theorem~\ref{stable_manifold}, Theorem~\ref{Local unstable manifolds}, Theorem~\ref{center_manifold}, and Corollary~\ref{stability}. Finally, we conclude the paper with examples that illustrate the applicability of our findings.

	\subsection*{Preliminaries and notation}
	
	In this section, we gather some conventions, notation, and basic definitions, which will be used for the rest of the paper.
	\begin{itemize}
		\item  For all finite Banach spaces, we will use the same notation $\Vert .\Vert$ to denote the norm. Also, For two Banach spaces $U$ and $V$, by $\mathcal{L}(U, V)$, we mean all bounded linear functions from $U$ to $V$  with the usual operator norm.
		\item We will identify $\mathcal{L}(U,\mathcal{L}(V,W))$ with $\mathcal{L}(U\otimes V,W))$ .
		\item  By $C^n_b(V,W)$, we mean the space of bounded functions $G \colon V \to W$ having $n$ bounded derivatives.
		\item  Assume  $I$ is an interval in $\R$ and $U$ be a finite Banach space. A map $ \xi: I\rightarrow U $, will also be called a \emph{path}. For $\xi$, we denote its increment by $ \delta \xi_{s,t}=\xi_{t}-\xi_{s} $ where by $ \xi_{t} $ we mean $ \xi(t) $. We set
		\begin{align*}
			\Vert \xi\Vert_{\infty,I}:=\sup_{s\in I}\Vert \xi_{s}\Vert.
		\end{align*}
		For  $\gamma \in (0,1]$, we define the $ \gamma$-H\"older seminorm by
		\begin{align*}
			\Vert \xi\Vert_{\gamma, I} := \sup_{\substack {s,t\in I\\s\neq t  }} \frac{\Vert \delta\xi_{s,t}\Vert}{\vert t-s\vert^{\gamma}}.
		\end{align*}
		Also, we set $\Vert \xi\Vert_{C^\gamma, I}=\max\lbrace \Vert \xi\Vert_{\infty,I},\Vert \xi\Vert_{\gamma,I}\rbrace$.
		\item Assume $A\in \mathcal{L}(U,W)$ and $B\in U$. By $AB$ we mean $A\circ B\in W$, i.e. the composition of $A$ and $B$. 
		\item  We call $V\colon \mathbb{R}^m\rightarrow \mathcal{L}(\mathbb{R}^d,\mathbb{R}^m)$ a $\text{Lip}^{p}$-vector field, if $V$ be $\lfloor p\rfloor $-times continuously differentiable with bounded derivatives and for $r=p-\lfloor p\rfloor$
		\begin{align}\label{XCVCV}
			\ \ \sup_{\substack{z_1,z_0\in\mathbb{R}^m\\
					z_1\neq z_0}}\frac{\Vert D^{\lfloor p\rfloor }_{z_1}V-D^{\lfloor p\rfloor }_{z_0}V\Vert}{\Vert z_1-z_0\Vert^r}<\infty.
		\end{align}
		\item We say $V_{0}:\mathbb{R}^{m}\rightarrow  \mathbb{R}^m$ is a locally Lipschitz continuous vector field with linear growth on $\mathbb{R}^m$, if there are constants $\kappa_1,\kappa_2\geq 0$ such that
		\begin{align}\label{bound-drift}
			\Vert V_0(z)\Vert\leq \kappa_1+\kappa_{2}\Vert z\Vert,
		\end{align}  
		for every $z \in \R^m$.
		\item  Assume that $X \colon \mathbb{R} \to \mathbb{R}^d$ is a locally $\gamma$-H\"older path with $\frac{1}{4}<\gamma\leq \frac{1}{2}$.
			Suppose there exist continuous functions
			\begin{align*}
				\mathbb{X}^{2} \colon \mathbb{R} \times \mathbb{R} \to (\mathbb{R}^{d})^{\otimes 2}, \quad
				\mathbb{X}^{3} \colon \mathbb{R} \times \mathbb{R} \to (\mathbb{R}^{d})^{\otimes 3},
			\end{align*}
			called the second and third \emph{L\'evy areas} of $X$, respectively, which satisfy the following properties:
			\begin{itemize}
				\item For every $s,u,t \in \mathbb{R}$,
				\begin{align*}
					\mathbb{X}_{s,t}^{2}
					&=
					\mathbb{X}_{s,u}^{2}
					+
					\mathbb{X}_{u,t}^{2}
					+
					\delta X_{s,u}
					\otimes
					\delta X_{u,t}, \\
					\mathbb{X}_{s,t}^{3}
					&=
					\mathbb{X}_{s,u}^{3}
					+
					\mathbb{X}_{u,t}^{3}
					+
					\mathbb{X}_{s,u}^{2}
					\otimes
					\delta X_{u,t}
					+
					\delta X_{s,u}
					\otimes
					\mathbb{X}_{u,t}^{2}.
				\end{align*}
				\item Furthermore, these functions satisfy the analytic conditions
				\[
				\| \mathbb{X}^{2} \|_{2\gamma , I}
				:=
				\sup_{\substack{s,t\in I\\ s\neq t}}
				\frac{
					\Vert\mathbb{X}^{2}_{s,t}\Vert
				}{
					|t-s|^{2\gamma}
				}
				< \infty,
				\quad
				\| \mathbb{X}^{3} \|_{3\gamma , I}
				:=
				\sup_{\substack{s,t\in I\\ s\neq t}}
				\frac{
					\Vert\mathbb{X}^{3}_{s,t}\Vert
				}{
					|t-s|^{3\gamma}
				}
				< \infty
				\]
				for every compact interval $I \subset \mathbb{R}$.
			\end{itemize}
			In this case, we define $\mathbf{X} = \big( X, \mathbb{X}^{2}, \mathbb{X}^{3} \big)$ as a \emph{$\gamma$-rough path}, where
			\[
			\mathbf{X} \colon \mathbb{R} \to T^{3}(\mathbb{R}^{d})
			:=
			\mathbb{R}
			\oplus
			\mathbb{R}^{d}
			\oplus
			(\mathbb{R}^{d})^{\otimes 2}
			\oplus
			(\mathbb{R}^{d})^{\otimes 3},
			\]
			is given by
			\[
			t \mapsto
			1
			\oplus
			X_{t}
			\oplus
			\mathbb{X}_{t}^{2}
			\oplus
			\mathbb{X}_{0,t}^{3}.
			\]
			Using the increment notation
			\[
			\mathbf{X}_{s,t}
			:=
			\mathbf{X}_{s}^{-1}
			\otimes
			\mathbf{X}_{t},
			\]
			we obtain
			\[
			\mathbf{X}_{s,t}
			=
			1
			\oplus
			\delta X_{s,t}
			\oplus
			\mathbb{X}_{s,t}^{2}
			\oplus
			\mathbb{X}_{s,t}^{3}.
			\]
			Finally, we define the rough path norm as
			\begin{align}
				\Vert\mathbf{X}\Vert_{\gamma,I}
				:=
				\max \Big\{
				\Vert X \Vert_{\gamma,I},
				\sqrt{ \Vert \mathbb{X}^{2} \Vert_{2\gamma,I} },
				\sqrt[3]{ \Vert \mathbb{X}^{3} \Vert_{3\gamma,I} }
				\Big\}.
			\end{align}
		\item    
		We call a \emph{$\gamma$-rough path} \(\mathbf{X} = \big(X, \mathbb{X}^2, \mathbb{X}^3\big)\) \emph{geometric} if it can be approximated by smooth paths and their iterated integrals; see \cite[Definition~9.15]{FV10} for more details.
		In particular, \(\mathbf{X}\) takes values in the free nilpotent group of order $3$. Thus, if
		\[
		X_t = \sum_{i=1}^d X_t^i e_i, \quad
		\mathbb{X}_{s,t}^2 = \sum_{i,j=1}^d \mathbb{X}_{s,t}^{i,j} \, e_i \otimes e_j, \quad
		\mathbb{X}_{s,t}^3 = \sum_{i,j,k=1}^d \mathbb{X}_{s,t}^{i,j,k} \, e_i \otimes e_j \otimes e_k,
		\]
		then the following algebraic constraints must hold for every \( (i,j,k) \in \{1,2,\ldots,d\}^3 \) and all \( s, t \in \mathbb{R} \):
		\begin{align*}
			\mathbb{X}^{i,j}_{s,t} + \mathbb{X}^{j,i}_{s,t}
			&= \delta X^{i}_{s,t} \, \delta X^{j}_{s,t}, \\
			\sum_{\sigma \in S_3} \mathbb{X}_{s,t}^{\sigma(i), \sigma(j), \sigma(k)} 
			&= \delta X^i_{s,t} \, \delta X^j_{s,t} \, \delta X^k_{s,t},
		\end{align*}
		where \( S_3 \) denotes the set of all permutations of the triple \( (i, j, k) \), and \( (e_i)_{1 \leq i \leq d} \) is the standard basis of \( \mathbb{R}^d \).
		\item By \( a \lesssim b \), we mean that there exists a constant \( C \), independent of \( \mathbf{X} \), such that \( a \leq Cb \).
	\end{itemize}
	\section{Rough differential equations}\label{sec:RDE}
	In this section, we consider the rough differential equation of the form 
	\begin{align}\label{SDE}
		\mathrm{d}Z_{t} = V(Z_t) \,\mathrm{d} \mathbf{X}_t + V_{0}(Z_t) \, \mathrm{d}t, \quad Z_{0} = z_0 \in \mathbb{R}^m.
	\end{align}
	We will study various aspects of this equation. In particular, we will derive several estimates that are essential for analyzing the long-term behavior of the solutions to this type of equation. We will adopt the following assumption until the end of this section.
	\begin{assumption}\label{ASSVVNN}
		\begin{itemize}
			\item 	We assume $\frac{1}{4}<\gamma \leqslant \frac{1}{2}$ and that $\mathbf{X} = (X, \mathbb{X}^2, \mathbb{X}^3)$ is a  $\gamma$-Hölder \emph{geometric rough path} where $X$ takes values in $\mathbb{R}^d$.
			\item  We assume $V_{0}:\mathbb{R}^{m}\rightarrow  \mathbb{R}^m$ is a $C^{1}$-vector field such that we can find a polynomial $P$ for which
			\begin{align}\label{BNM<<<}
				\forall z\in\mathbb{R}^m:\ \ \  \Vert V_{0}(z)\Vert\leq P(\Vert z\Vert).
			\end{align} 
			\item For $ \frac{1}{\gamma}+1<p\leq \lfloor \frac{1}{\gamma}\rfloor+2 $, $V\colon \mathbb{R}^m\rightarrow \mathcal{L}(\mathbb{R}^d,\mathbb{R}^m)$ is a $\text{Lip}^{p}-$vector field. 
			\item We that for every initial value $z_0\in\mathbb{R}^m$, the equation \eqref{SDE} admits a unique solution such that we can find a polynomial $R$ and a continuous function $\Phi$ for which
			\begin{align}\label{sup_norm}
				\forall T> 0:\ \ \Vert\phi_{\mathbf{X}}(z_0)\Vert_{C^\gamma,[0,T]} \leq\Phi(T) R(\Vert z_0\Vert,\Vert\mathbf{X}\Vert_{\gamma,[0,T]}).
			\end{align}
		\end{itemize}
	\end{assumption}
	We now state three examples in which \eqref{SDE} possesses unique solutions for every initial condition and in which the bound \eqref{sup_norm} holds.
	\begin{example}
		\begin{itemize}
			\item 
			The simplest case arises when \( V \) and \( V_0 \) are \(\text{Lip}^p\)-vector fields. It is well-known that, for every \( z_0 \in \mathbb{R}^m \), the equation \eqref{SDE} admits a unique solution \cite[Theorem 10.26]{FV10}, which is also differentiable with respect to the initial condition \cite[Theorem 11.6]{FV10}. Moreover, Assumption \eqref{ASSVVNN} is satisfied in this setting.
			\item Assume that \( V_0 \) is a locally Lipschitz continuous vector field with linear growth on \( \mathbb{R}^m \). Then, by \cite[Theorem 3.1]{RS17}, the a priori bound \eqref{sup_norm} holds. Consequently, Assumption \eqref{ASSVVNN} is satisfied.
			\item If \( V_0 \) satisfies the one-sided conditions formulated in \cite[Equations (4.2) and (4.3)]{RS17}, it is shown in \cite[Theorem 4.3]{RS17} that the a priori bound \eqref{sup_norm} still holds. \footnote{Strictly speaking, the bounds in \cite[Theorem 3.1]{RS17} and \cite[Theorem 4.3]{RS17} are formulated for the $p$-variation and not for the H\"older-norm. However, an inspection of the proof reveals that the same strategy applied there works also for the H\"older norm provided the equations are driven by H\"older rough paths.} 
		\end{itemize}
	\end{example}
	\subsection{ Basic objects, definitions and auxiliary Lemmas}
	In this section, we introduce key concepts from rough path theory and define integrals within this framework. Furthermore, we establish several auxiliary lemmas that will be instrumental in the subsequent section.
	\begin{remark}
		If we assume \( \frac{1}{3} < \gamma \leq \frac{1}{2} \), the calculations simplify significantly, and the third \emph{Lévy area} becomes superfluous. In this paper, however, we address the more challenging case, namely \( \frac{1}{4} < \gamma \leq \frac{1}{3} \). All statements remain valid for \( \frac{1}{3} < \gamma \leq \frac{1}{2} \) under Assumption~\ref{ASSVVNN}.
	\end{remark}

	
	To facilitate the forthcoming calculations, it is helpful to recall the definition of a path \emph{controlled by \( \mathbf{X} \)} as introduced in \cite[Section 4.5]{FH20}.
	\begin{definition}\label{DDCD}
		Let \(\gamma_1 \leq \gamma\) and \(Y: [a,b] \rightarrow W\) be a \(\gamma_1\)-Hölder path taking values in some finite-dimensional Banach space \(W\). If there are \(\gamma_1\)-Hölder paths \(Y^{(1)}\) and \(Y^{(2)}\) taking values in \(\mathcal{L}(\mathbb{R}^d,W)\) and \(\mathcal{L}(\mathbb{R}^d,\mathcal{L}(\mathbb{R}^d,W)) \approx \mathcal{L}(\mathbb{R}^d \otimes \mathbb{R}^d,W)\) that satisfy
		\begin{align}\label{expansion}
			\begin{split}
				s, t \in [a,b], \; s \leq t\colon \quad \delta Y_{s,t} - Y^{(1)}_s \delta X_{s,t} - Y^{(2)}_s \mathbb{X}^{2}_{s,t} &= Y^{\#}_{s,t} = \mathcal{O}(|t-s|^{3\gamma}) \quad \text{and} \\
				\delta Y^{(1)}_{s,t} - Y^{(2)}_s \delta X_{s,t} &= (Y^{(1)})^{\#}_{s,t} = \mathcal{O}(|t-s|^{2\gamma}),
			\end{split}
		\end{align}
		the triple \((Y,Y^{(1)},Y^{(2)})\) is said to be a \emph{path controlled by \(\mathbf{X}\)}. We denote \(Y^{(1)}\) and \(Y^{(2)}\) as the first and second Gubinelli derivatives, respectively, and \(Y^{\#}\) and \((Y^{(1)})^{\#}\) as the remainder terms. We use \(\mathcal{D}^{\gamma_1}_{\mathbf{X},W}([a,b])\) to denote the space of controlled paths. We also impose the following norm:
		\begin{align}\label{NM(87654)}
			\begin{split}
				&\vertiii{(Y,Y^{(1)},Y^{(2)})}_{\mathcal{D}^{\gamma_1}_{\mathbf{X},W}([a,b])} 
				\\&:= \max \bigg\lbrace \| Y \|_{\infty,[a,b]}, \| Y^{(1)} \|_{\infty,[a,b]}, \| Y^{(2)} \|_{\infty,[a,b]}, \| Y^{(2)} \|_{\gamma_1,[a,b]}, \| (Y^{(1)})^{\#} \|_{2\gamma_1,[a,b]}, \| Y^{\#} \|_{3\gamma_1,[a,b]} \bigg\rbrace.
			\end{split}
		\end{align}
		It is obvious that if $(Z,Z^{(1)},Z^{(2)})$ be another path controlled by $\mathbf{X}$, then 
		\begin{align*}
			&\vertiii{(Y+Z,Y^{(1)}+Z^{(1)},Y^{(2)}+Z^{(2)})}_{\mathcal{D}^{\gamma_1}_{\mathbf{X},W}([a,b])}\\&\qquad\leq \vertiii{(Y,Y^{(1)},Y^{(2)})}_{\mathcal{D}^{\gamma_1}_{\mathbf{X},W}([a,b])}+\vertiii{(Z,Z^{(1)},Z^{(2)})}_{\mathcal{D}^{\gamma_1}_{\mathbf{X},W}([a,b])}.
		\end{align*}
	\end{definition}
	\begin{remark}\label{VCYT^&*}
		Following estimates are  direct consequences of \eqref{expansion} and \eqref{NM(87654)}
		\begin{align*}
			&\Vert Y\Vert_{\gamma_1,[a,b]}\lesssim (1+\Vert\mathbf{X}\Vert_{\gamma,[a,b]}^2)\Vert Y\Vert_{\mathcal{D}^{\gamma_1}_{\mathbf{X},W}([a,b])},\\
			&\Vert Y^{1}\Vert_{\gamma_1,[a,b]}\lesssim (1+\Vert\mathbf{X}\Vert_{\gamma,[a,b]})\Vert Y\Vert_{\mathcal{D}^{\gamma_1}_{\mathbf{X},W}([a,b])}.
		\end{align*}
	\end{remark} 
	\begin{remark}
		To simplify notation, we will occasionally omit $W$ and use $\mathcal{D}^{\gamma_1}_{\mathbf{X}}([a,b])$ in place of $\mathcal{D}^{\gamma_1}_{\mathbf{X},W}([a,b])$.
	\end{remark}
	\begin{remark}
		From \eqref{expansion}, 
		\begin{align*}
			s,u,t \in [a,b], \; s \leq u\leq t\colon\ \  &Y^{\#}_{s,t}=Y^{\#}_{s,u}+Y^{\#}_{u,t}+(Y^{(1)})^{\#}_{s,u}\delta X_{u,t}+\delta Y^{(2)}_{s,u}\mathbb{X}^{2}_{u,t},\\
			& (Y^{(1)})^{\#}_{s,t}=(Y^{(1)})^{\#}_{s,u}+(Y^{(1)})^{\#}_{u,t}+\delta Y^{(2)}_{s,u}\delta X_{u,t}.
		\end{align*}
		This yields the following inequality
		\begin{align}\label{DDFFD}
			\begin{split}
				&\vertiii{(Y,Y^{(1)},Y^{(2)})}_{\mathcal{D}^{\gamma_1}_{\mathbf{X}}([a,c])}\leq  (\Vert\mathbf{X}\Vert_{\gamma,[a,c]}+1)^2\vertiii{(Y,Y^{(1)},Y^{(2)})}_{\mathcal{D}^{\gamma_1}_{\mathbf{X}}([a,b])}\\&\quad+\vertiii{(Y,Y^{(1)},Y^{(2)})}_{\mathcal{D}^{\gamma_1}_{\mathbf{X}}([b,c])},\\&\quad
			\end{split}
		\end{align}
		where $a<b<c$ and $c-a\leq 1$.
	\end{remark}
	We establish the following lemma for the composition of two controlled paths.	\begin{lemma}\label{composit}
		Let $U$ and $W$ be two finite-dimensional Banach spaces, and assume $\gamma_1 \leq \gamma$. Let $A$ be a $\gamma_1$-Hölder path taking values in $U$, with $(A, A^{(1)}, A^{(2)})$ controlled by $\mathbf{X}$. Additionally, assume that $B$ is a $\gamma_1$-Hölder path taking values in $\mathcal{L}(U, W)$, and that $(B, B^{(1)}, B^{(2)})$ is also controlled by $\mathbf{X}$. Define the symmetric operator:
		\begin{align*}
			&\text{Sym} : \mathbb{R}^d \otimes \mathbb{R}^d \rightarrow \mathbb{R}^d \otimes \mathbb{R}^d, \\
			&\text{Sym}(v_1 \otimes v_2) = \frac{v_1 \otimes v_2 + v_2 \otimes v_1}{2}.
		\end{align*}
		Set
		\begin{align*}
			(A B)^{(1)}_{s} &= A^{(1)}_{s} B_{s} + A_{s} B^{(1)}_{s}, \\
			(A B)^{(2)}_{s} &= A^{(2)}_{s} B_{s} + A_{s} B^{(2)}_{s} + 2(A^{(1)}_{s} B^{(1)}_{s}) \circ \text{Sym},
		\end{align*}
		with the following actions:
		\begin{align*}
			&v \in \mathbb{R}^d: \quad (A B)^{(1)}_{s}(v) = A^{(1)}_{s}(v) B_{s} + A_{s} B^{(1)}_{s}(v), \\
			&v_1 \otimes v_2 \in \mathbb{R}^d \otimes \mathbb{R}^d : \\
			&\quad (A^{(1)}_s B^{(1)}_{s})(v_1 \otimes v_2) = A^{(1)}_{s}(v_1) B^{(1)}_{s}(v_2), \\
			&\quad (A B)^{(2)}_{s}(v_1 \otimes v_2) := A^{(2)}_{s}(v_1 \otimes v_2) B_{s} + A_{s} B^{(2)}_{s}(v_1 \otimes v_2) \\
			&\qquad\qquad\quad + A^{(1)}_{s}(v_1) B^{(1)}_{s}(v_2) + A^{(1)}_{s}(v_2) B^{(1)}_{s}(v_1).
		\end{align*}
		Also,
		\begin{align*}
			(A B)^{\#}_{s,t} &= A^{\#}_{s,t} B_{s} + A_{s} B^{\#}_{s,t} 
			+ (A^{(1)}_{s} (\delta X_{s,t}))(B^{(2)}_{s} \mathbb{X}_{s,t} + B^{\#}_{s,t}) 
			+ (A^{(2)}_{s} \mathbb{X}_{s,t} + A^{\#}_{s,t})(\delta B_{s,t}),
			\\
			((A B)^{(1)})^{\#}_{s,t} &= (A^{(1)})^{\#}_{s,t} B_{s} + A_{s} (B^{(1)})^{\#}_{s,t} 
			+ A_{s}^{(1)} (B^{(1)})^{\#}_{s,t} + (A^{(1)})^{\#}_{s,t} B_{s}^{(1)}
			\\
			&\quad + (\delta A^{(1)}_{s,t})(\delta B_{s,t}) + (\delta A_{s,t})(\delta B^{(1)}_{s,t}).
		\end{align*}
		Then $(AB, (A B)^{(1)}, (A B)^{(2)})$ is a path controlled by $\mathbf{X}$. In addition, $(A B)^{\#}$ and $((A B)^{(1)})^{\#}$ are the remainder terms as in \eqref{expansion}. Furthermore, 
		\begin{align}\label{CDCa}
			\vertiii{(AB,(A B)^{(1)},(A B)^{(2)})}_{\mathcal{D}^{\gamma_1}_{\mathbf{X}}([a,b])}\lesssim (1+\Vert\mathbf{X}\Vert_{\gamma,[a,b]}^4)\vertiii{A}_{\mathcal{D}^{\gamma_1}_{\mathbf{X}}([a,b])}\vertiii{B}_{\mathcal{D}^{\gamma_1}_{\mathbf{X}}([a,b])}.
		\end{align}
	\end{lemma}	
	\begin{proof}
		Note that
		\begin{align}\label{MKLa}
			\delta(AB)_{s,t}=A_{s}(\delta B_{s,t})+(\delta A_{s,t})B_s+(\delta A_{s,t})(\delta B_{s,t}).
		\end{align} 
		For the remainder of the proof, it is sufficient to substitute the expansions of \( A \) and \( B \) into \eqref{MKLa}. Notably, since \( \mathbf{X} \) is geometric, we have
		\begin{align*}
			2(A^{(1)}_{s} B^{(1)}_{s})\circ{Sym}(\mathbb{X}_{s,t})=(A^{(1)}_sB^{(1)}_s)\delta X_{s,t}\otimes\delta X_{s,t}.
		\end{align*}
		Also, \eqref{CDCa} follows from Remark \ref{VCYT^&*} and the expressions that we provide in the statement of Lemma. 
	\end{proof}
	\begin{remark}
		If the meaning of the first and second Gubinelli derivatives is clear from the context, we may adopt the convention of \( \vertiii{Y}_{\mathcal{D}^{\gamma}_{\mathbf{X},W}([a,b])} \) instead of \( \vertiii{(Y,Y^{(1)},Y^{(2)})}_{\mathcal{D}^{\gamma_1}_{\mathbf{X},W}([a,b])} \).
	\end{remark}
	It is well-known that the composition of a controlled path by \( \mathbf{X} \) with a smooth function results in a path that is again controlled by \( \mathbf{X} \). In the following lemma, we make this statement precise and perform explicit calculations for our future purposes.
	\begin{lemma}\label{expansion_2}
		Let $W$ and $U$ be finite-dimensional Banach spaces, $\gamma_1\leq \gamma$ and $G \colon W \to U$ be $3$-times Fr\'echet differentiable. Assume  $(Y,Y^{(1)},Y^{(2)})\in \mathcal{D}^{\gamma_1}_{\mathbf{X}}([a,b])$. Then $(G(Y),G(Y)^{(1)},G(Y)^{(2)})\in \mathcal{D}^{\gamma_1}_{\mathbf{X}}([a,b])$, where
		\begin{align}\label{T1}
			\begin{split}
				&	G(Y)^{(1)} = D_YG(Y^{(1)}),\\
				&v\in\mathbb{R}^d: \ \	G(Y)^{(1)} v=D_YG(Y^{(1)} v)
			\end{split}
		\end{align}
		and
		\begin{align}\label{T2}
			\begin{split}
				&	G(Y)^{(2)} = D_{Y}^2G(Y^{(1)},Y^{(1)}) + D_YG (Y^{(2)}),	\\
				&v_1\otimes v_2\in\mathbb{R}^d\otimes\mathbb{R}^d: \ \ G(Y)^{(2)}(v_1\otimes v_2)= D_{Y}^2G(Y^{(1)}v_1,Y^{(1)}v_2)+	D_YG (Y^{(2)}(v_1\otimes v_2)).
			\end{split}
		\end{align}
		Furthermore,
		\begin{align}\label{T3}
			\begin{split}
				(G(Y))^{\#}_{s,t}&=D_{Y_s}G(Y^{\#}_{s,t})+ \frac{1}{2}D_{Y_s}^2G\big(Y^{(1)}_{s}\delta X_{s,t},Y^{(2)}_s\mathbb{X}_{s,t}^2 +Y^{\#}_{s,t}\big)+\frac{1}{2}D^{2}_{Y_s}G\big(Y^{(2)}_s\mathbb{X}_{s,t}^2 +Y^{\#}_{s,t},\delta Y_{s,t}\big)\\&+ \int_{0}^{1}\frac{(1-\sigma)^2}{2}D^{3}_{\sigma Y_t+(1-\sigma)Y_s}G\big(\delta Y_{s,t},\delta Y_{s,t},\delta Y_{s,t}\big) \, \mathrm{d}\sigma,
				\\ \big(G(Y)^{(1)}\big)^{\#}_{s,t}&= D_{Y_s}^{2} G\big(Y^{(2)}_s\mathbb{X}^{2}_{s,t}+Y^{\#}_{s,t},Y^{(1)}_{s}\big)+D_{Y_s}G\big((Y^{(1)})^{\#}_{s,t}\big)\\+&\int_{0}^{1}(1-\sigma)D^{3}_{\sigma Y_t+(1-\sigma)Y_s}G\big(\delta Y_{s,t},\delta Y_{s,t},Y^{(1)}_{s}\big)\mathrm{d}\sigma+(D_{Y_t}G-D_{Y_s}G)(\delta Y_{s,t}),
			\end{split}
		\end{align}
		with the following action
		\begin{align*}
			v\in\mathbb{R}^d: \  \ &\big(G(Y)^{(1)}\big)^{\#}_{s,t}v= D_{Y_s}^{2} G\big(Y^{(2)}_s\mathbb{X}^{2}_{s,t}+Y^{\#}_{s,t},Y^{(1)}_{s}v\big)+D_{Y_s}G\big((Y^{(1)})^{\#}_{s,t}v\big)\\+&\int_{0}^{1}(1-\sigma)D^{3}_{\sigma Y_t+(1-\sigma)Y_s}G\big(\delta Y_{s,t},\delta Y_{s,t},Y^{(1)}_{s}v\big)\mathrm{d}\sigma+(D_{Y_t}G-D_{Y_s}G)(\delta Y_{s,t})v.
		\end{align*}
		If $G\in C_{b}^3(W,U)$, then
		\begin{align}\label{composite}
			\begin{split}
				u,v\in [a,b], \;  u\leq v\colon \  &\vertiii{\big(G(Y),G(Y)^{(1)},G(Y)^{(2)}\big)}_{\mathcal{D}^{\gamma_1}_{\mathbf{X}}([u,v])}\\&\lesssim\big(1+ \vertiii{Y}_{\mathcal{D}^{\gamma_1}_{\mathbf{X}}([u,v])}^9+\Vert\mathbf{X}\Vert_{\gamma,[u,v]}^9\big).
			\end{split}
		\end{align}
	\end{lemma}
	\begin{proof}
		The expressions for the derivatives and remainder terms follow from the Taylor expansions given below, together with the fact that \(\mathbf{X}\) is geometric:
		\begin{align}
			\label{DecoU}
			\begin{split}
				&\forall x, y \in W: \\
				G(y) - G(x) &= D_x G(y - x) + \int_0^1 (1 - \sigma) D^2_{\sigma y + (1 - \sigma)x} G(y - x, y - x) \, \d \sigma, \\
				G(y) - G(x) &= D_x G(y - x) + \frac{1}{2} D^2_x G(y - x, y - x) \\
				&\quad + \frac{1}{2} \int_0^1 \frac{(1 - \sigma)^2}{2} D^3_{\sigma y + (1 - \sigma)x} G(y - x, y - x, y - x) \, \d \sigma.
			\end{split}
		\end{align}
		Additionally, when \( G \in C_{b}^3(W, U) \), the inequality \eqref{composite} follows directly from Remark \ref{VCYT^&*}, \eqref{T1}, \eqref{T2}, \eqref{T3}, and the following simple inequality:
		\begin{align}
			x, y \geq 0: \quad \max_{\substack{i,j \geq 0 \\ i + j \leq 9}} \{ x^i y^j \} \lesssim 1 + x^9 + y^9.
		\end{align}
	\end{proof}
	From the latter lemma, we obtain the following result:
	\begin{corollary}
		Assume $0 < r \leq 1$ and that $G$ is a $\text{Lip}^{3+r}$ vector field. Let $(Y, Y^{(1)}, Y^{(2)})$ and $(Z, Z^{(1)}, Z^{(2)})$ belong to $\mathcal{D}^{\gamma}_{\mathbf{X}}([a, b])$. Then, for a fixed $0 < \kappa < 1$,
		\begin{align}\label{bNNNMM}
			\begin{split}
				u,v\in [a,b], \; u\leq v\colon\ \ &\vertiii{G(Y)-G(Z)}_{\mathcal{D}^{\kappa\gamma}_{\mathbf{X}}([u,v])}\\&\lesssim\max\left\lbrace \vertiii{Y-Z}_{\mathcal{D}^{\gamma}_{\mathbf{X}}([u,v])}^r,\vertiii{Y-Z}_{\mathcal{D}^{\gamma}_{\mathbf{X}}([u,v])},\vertiii{Y-Z}_{\mathcal{D}^{\gamma}_{\mathbf{X}}([u,v])}^{1-\kappa} \right\rbrace\\&\quad\ \times\left(1+ \vertiii{Y}_{\mathcal{D}^{\gamma}_{\mathbf{X}}([u,v])}^9+\vertiii{Z}_{\mathcal{D}^{\gamma}_{\mathbf{X}}([u,v])}^9+\Vert\mathbf{X}\Vert_{\gamma,[u,v]}^9\right),
			\end{split}
		\end{align}	
	\end{corollary}
	\begin{proof}
		First, note that from Lemma \ref{expansion_2},
		\begin{align*}
			&(G(Y)^{(2)})_t-  (G(Z)^{(2)})_t=D_{Y_t}^2G(Y^{(1)}_t,Y^{(1)}_t) + D_{Y_{t}}G (Y^{(2)}_t)-D_{Z_t}^2G(Z^{(1)}_t,Z^{(1)}_t) - D_{Z_{t}}G (Z^{(2)}_t)
			\\& =\int_{0}^{1}D^{3}_{\sigma Y_t+(1-\sigma)Z_t}G(Y_t-Z_t,Y^{(1)}_t,Y^{(1)}_t)\mathrm{d}\sigma+\big(D^{2}_{Z_t}G(Y^{(1)}_t,Y^{(1)}_t)-D_{Z_t}^2G(Z^{(1)}_t,Z^{(1)}_t)\big)\\&+\big(D_{Y_{t}}G (Y^{(2)}_t)-D_{Z_{t}}G (Z^{(2)}_t)\big)
		\end{align*}
		Therefore,
		\begin{align}\label{V1}
			\Vert G(Y)^{(2)}-G(Z)^{(2)}\Vert_{\infty,[a,b]}\lesssim\vertiii{Y-Z}_{\mathcal{D}^{\gamma}_{\mathbf{X}}([a,b])}(1+\vertiii{Y}^{2}_{\mathcal{D}^{\gamma}_{\mathbf{X}}([a,b])}+\vertiii{Z}^{2}_{\mathcal{D}^{\gamma}_{\mathbf{X}}([a,b])})
		\end{align}
		Also, from Remark \ref{VCYT^&*} 
		\begin{align*}
			&\Vert \delta (G(Y)^{(2)}-G(Z)^{(2)})_{s,t}\Vert\leq \Vert\delta(G(Y)^{(2)})_{s,t} \Vert+\Vert\delta(G(Z)^{(2)})_{s,t} \Vert\\&\lesssim (t-s)^{\gamma}(1+\Vert\mathbf{X}\Vert_{\gamma,[a,b]}^2)(1+\vertiii{Y}^{2}_{\mathcal{D}^{\gamma}_{\mathbf{X}}([a,b])}+\vertiii{Z}^{2}_{\mathcal{D}^{\gamma}_{\mathbf{X}}([a,b])}).
		\end{align*}
		Consequently,
		\begin{align}\label{V2}
			\begin{split}
				&\Vert \delta (G(Y)^{(2)}-G(Z)^{(2)})_{s,t}\Vert\lesssim \Vert \delta (G(Y)^{(2)}-G(Z)^{(2)})_{s,t}\Vert^{\kappa}\Vert G(Y)^{(2)}-G(Z)^{(2)}\Vert_{\infty,[a,b]}^{1-\kappa}\\&\quad\lesssim (t-s)^{\kappa\gamma}\vertiii{Y-Z}_{\mathcal{D}^{\gamma}_{\mathbf{X}}([a,b])}^{1-\kappa} (1+\Vert\mathbf{X}\Vert_{\gamma,[a,b]}^2)(1+\vertiii{Y}^{2}_{\mathcal{D}^{\gamma}_{\mathbf{X}}([a,b])}+\vertiii{Z}^{2}_{\mathcal{D}^{\gamma}_{\mathbf{X}}([a,b])}).
			\end{split}
		\end{align}
		From the assumption on \(G\),
		\begin{align}\label{V3}
			\begin{split}
				&\max\left\lbrace  \Vert G(Y)-G(Z)\Vert_{\infty,[u,v]} \Vert G(Y)^{(1)}-G(Z)^{(1)}\Vert_{\infty,[a,b]},\Vert G(Y)^{(2)}-G(Z)^{(2)}\Vert_{\infty,[a,b]}\right\rbrace\\
				& \lesssim\vertiii{Y-Z}_{\mathcal{D}^{\gamma}_{\mathbf{X}}([u,v])}(1+\vertiii{Y}^{2}_{\mathcal{D}^{\gamma}_{\mathbf{X}}([a,b])}+\vertiii{Z}^{2}_{\mathcal{D}^{\gamma}_{\mathbf{X}}([a,b])}).
			\end{split}
		\end{align} 
		Since \(G\) is a \(\mathrm{Lip}^{3+r}\)-vector field, it follows from \eqref{T3} and Remark~\ref{VCYT^&*} that
		\begin{align}\label{v4}
			\begin{split}
				&\max\left\lbrace \Vert (G(Y))^{\#}-(G(Z))^{\#}\Vert\rbrace_{3\kappa\gamma,[u,v]},\Vert \big(G(Y)^{(1)}\big)^{\#}-\big(G(Z)^{(1)}\big)^{\#}\Vert_{2\kappa\gamma,[u,v]}\right\rbrace\\
				&\lesssim\max\left\lbrace \vertiii{Y-Z}_{\mathcal{D}^{\gamma}_{\mathbf{X}}([u,v])}^r,\vertiii{Y-Z}_{\mathcal{D}^{\gamma}_{\mathbf{X}}([u,v])} \right\rbrace\left(1+ \vertiii{Y}_{\mathcal{D}^{\gamma}_{\mathbf{X}}([u,v])}^9+\vertiii{Z}_{\mathcal{D}^{\gamma}_{\mathbf{X}}([u,v])}^9+\Vert\mathbf{X}\Vert_{\gamma,[u,v]}^9\right).
			\end{split}
		\end{align}
		The claim now follows from \eqref{V1}–\eqref{v4}.
	\end{proof}
	We can now define the integral with respect to \(\mathbf{X}\).
	\begin{lemma}
		Assume in Definition~\ref{DDCD} that \( W = \mathcal{L}(\mathbb{R}^d, \mathbb{R}^m) \) and \( 3\gamma_1 + \gamma > 1 \). Then, the rough integral
		\begin{align}\label{SEW}
			\int_{s}^{t}Y_{\tau}\mathrm{d}\mathbf{X}_{\tau}:=\lim_{\substack{|\pi|\rightarrow 0,\\ \pi=\lbrace s = \tau_0 < \tau_{1} < \ldots < \tau_{k}=t \rbrace}}\sum_{0\leq j<k}\big{[}Y_{\tau_j} (\delta X)_{\tau_j,\tau_{j+1}}+Y^{(1)}_{\tau_j}\mathbb{X}^{2}_{\tau_j,\tau_{j+1}}+Y^{(2)}_s\mathbb{X}^{3}\big{]}
		\end{align}
		exists. Furthermore  $\big(\int_{u}Y_{\tau}\mathrm{d}\mathbf{X}_{\tau},Y,Y^{(1)}\big)\in \mathcal{D}^{\gamma}_{\mathbf{X}}([u,b])$ and
		\begin{align}\label{SSWWEE}
			\begin{split}
				&u,v\in [a,b], \;  u\leq v\colon  \ \ \big\Vert\int_{u}^{v}Y_{\tau}\circ\mathrm{d}\mathbf{X}_{\tau}-Y_{u} (\delta X)_{u,v}-Y^{(1)}_{u}\mathbb{X}^{2}_{u,v}-Y^{(2)}_u\mathbb{X}^{3}_{u,v}\big\Vert\\& \lesssim
				\bigg[\Vert Y^{\#}\Vert_{3\gamma_1,[u,v]} \Vert X\Vert_{\gamma,[u,v]} + \Vert (Y^{(1)})^{\#}\Vert_{2\gamma_1,[u,v]} \big\Vert\mathbb{X}^{2}\big\Vert_{2\gamma,[u,v]}+  \Vert Y^{(2)}\Vert_{\gamma_1,[u,v]}\big\Vert\mathbb{X}^{3}\big\Vert_{3\gamma,[u,v]}\bigg] (v-u)^{3\gamma_1+\gamma}.
			\end{split}
		\end{align}
	\end{lemma}
	\begin{proof}
		The definition and well-definedness of \eqref{SEW} follow from the usual Sewing lemma \cite[Lemma 4.2]{FH20}. Additionally, \eqref{SSWWEE} is a direct consequence of the same lemma.
	\end{proof}
	Let us now make several remarks before proceeding to the next step.	\begin{remark}\label{BVCZRTGFV}
		Let \( u, v \in [a, b] \) with \( u \leq v \), and suppose that \( 3\gamma_1 > 2\gamma \). Then, from \eqref{SSWWEE}, we have:
		\begin{align}\label{NMUIO}
			\begin{split}
				& \vertiii{\big(\int_{u}Y_{\tau}\mathrm{d}\mathbf{X}_{\tau},Y,Y^{(1)}\big)}_{\mathcal{D}^{\gamma}_{\mathbf{X}}([u,v])}\lesssim \Vert Y_u\Vert+\Vert Y_{u}^{(1)}\Vert\Vert\mathbf{X}\Vert_{\gamma,[u,v]}^2+\Vert Y_{u}^{(2)}\Vert\Vert\mathbf{X}\Vert_{\gamma,[u,v]}^3\\&\quad+\vertiii{(Y,Y^{(1)},Y^{(2)})}_{\mathcal{D}^{\gamma_1}_{\mathbf{X}}([u,v])}(1+	\Vert\mathbf{X}\Vert_{\gamma,[u,v]}^3)(v-u)^{3\gamma_1-2\gamma}.
			\end{split}
		\end{align}
		In particular,
		\begin{align}\label{NMUIO_1}
			\vertiii{\big(\int_{u}Y_{\tau}\mathrm{d}\mathbf{X}_{\tau},Y,Y^{(1)}\big)}_{\mathcal{D}^{\gamma}_{\mathbf{X}}([u,v])}\lesssim \vertiii{(Y,Y^{(1)},Y^{(2)})}_{\mathcal{D}^{\gamma_1}_{\mathbf{X}}([u,v])}(1+	\Vert\mathbf{X}\Vert_{\gamma,[u,v]}^3).
		\end{align}
	\end{remark}
	\begin{remark}
		Recall that in equation \eqref{SDE}, the solution is given by
		\begin{align*}
			Z_t = Z_0 + \int_{0}^{t} V(Z_\tau) \, \mathrm{d} \mathbf{X}_\tau + \int_{0}^{t} V_{0}(Z_\tau) \, \mathrm{d}\tau.
		\end{align*}
		Here, \( \big(Z, V(Z), D_{Z}V(V(Z))\big) \in \mathcal{D}^{\gamma}_{\mathbf{X}} \), and the integral 
		\[
		\int_{0}^{t} V(Z_\tau) \, \mathrm{d}\mathbf{X}_\tau
		\]
		is defined in the sense of \eqref{SEW}.
	\end{remark}
		Recall that we assumed an a priori bound for the solution to  \eqref{SDE} in the \( C^{\gamma} \)-norm in \eqref{sup_norm}. However, it turns out that the \( C^{\gamma} \)-norm is not sufficient for our purposes. Thus, we conduct a more detailed analysis of this equation and derive an a priori bound in the \( \mathcal{D}^{\gamma}_{\mathbf{X}} \)-norm in the forthcoming proposition.

\begin{proposition}\label{RS17}
	Recall that Assumption~\ref{ASSVVNN} holds. Furthermore, suppose that \( \gamma \neq \frac{1}{3} \), and let \( (\phi^{t}_{\mathbf{X}}(z_0))_{t \geq 0} \) denote the solution to \eqref{SDE}. For \( s \leq t \), define
	\begin{align}\label{AAABB}
		\begin{split}
			\phi_{\mathbf{X}}(z_0)^{\#}_{s,t} &\coloneqq (\delta\phi_{\mathbf{X}}(z_0))_{s,t} - V(\phi^{s}_{\mathbf{X}}(z_0)) \delta X_{s,t} -D_{\phi^{s}_{\mathbf{X}}(z_0)} V\big( V(\phi^{s}_{\mathbf{X}}(z_0))\big) \mathbb{X}_{s,t}^2 \quad \text{and}\\ 
			\big(\phi_{\mathbf{X}}(z_0)^{(1)}\big)^{\#}_{s,t} &\coloneqq \big(\delta V(\phi_{\mathbf{X}}(z_0))\big)_{s,t}-D_{\phi^{s}_{\mathbf{X}}(z_0)}V \big( V(\phi^{s}_{\mathbf{X}}(z_0))\big)\delta X_{s,t}. \\
		\end{split}
	\end{align}
	Then 
	for the a polynomial $Q$  
	which depends on $ V$ and $V_0$, we have
	\begin{align}\label{BNNMM}
		&\vertiii{\phi_{\mathbf{X}}(z_0)}_{\mathcal{D}^{\gamma}_{\mathbf{X}}([0, T])} \leq  Q(\Vert z_0\Vert,\Vert\mathbf{X}\Vert_{\gamma,[0,T]}).
	\end{align}
\end{proposition}
\begin{proof}
	From \eqref{SDE}, \( \phi_{\mathbf{X}}(z_0) \) is controlled by \( \mathbf{X} \) and		\begin{align*}
		&\phi^{s}_{\mathbf{X}}(z_0)^{(1)} = V(\phi^{s}_{\mathbf{X}}(z_0)),\\
		&\phi^{s}_{\mathbf{X}}(z_0)^{(2)} = D_{\phi^{s}_{\mathbf{X}}(z_0)} V \big( V(\phi^{s}_{\mathbf{X}}(z_0))\big).
	\end{align*}
	Using the first identity in \eqref{DecoU} together with the expressions in \eqref{AAABB}, we obtain
	\begin{align}\label{VVVZZZ}
		\begin{split}
			\big(\phi_{\mathbf{X}}(z_0)^{(1)}\big)^{\#}_{s,t} &= D_{\phi^{s}_{\mathbf{X}}(z_0)}V\bigg(\phi_{\mathbf{X}}(z_0)^{\#}_{s,t} + \big(D_{\phi^{s}_{\mathbf{X}}(z_0)}V\big( V(\phi^{s}_{\mathbf{X}}(z_0))\big)\big)\mathbb{X}_{s,t}^2\bigg) \\
			&\quad +\int_{0}^{1}(1-\sigma) D^{2}_{\sigma \phi^{t}_{\mathbf{X}}(z_0)+(1-\sigma)\phi^{s}_{\mathbf{X}}(z_0)} V\big((\delta\phi_{\mathbf{X}}(z_0))_{s,t},(\delta\phi_{\mathbf{X}}(z_0))_{s,t}\big)\mathrm{d}\sigma.
		\end{split}
	\end{align}
	Thus, it is sufficient to find a polynomial bound for \( \|\phi_{\mathbf{X}}(z_0)^{\#}\|_{3\gamma,[0,T]} \). Recall that \( V: \mathbb{R}^m \to \mathcal{L}(\mathbb{R}^d, \mathbb{R}^m) \). We expand \( V(\phi^{s}_{\mathbf{X}}(z_0)) \) as in \eqref{expansion}. From Lemma~\ref{expansion_2}, 
	\begin{align}\label{expa_1}
		\begin{split}
			V(\phi^{s}_{\mathbf{X}}(z_0))^{(1)} &= D_{\phi^{s}_{\mathbf{X}}(z_0)}V\big(V(\phi^{s}_{\mathbf{X}}(z_0)\big),\\
			V(\phi^{s}_{\mathbf{X}}(z_0))^{(2)} &= D^{2}_{\phi^{s}_{\mathbf{X}}(z_0)}V\bigg(V\big(\phi^{s}_{\mathbf{X}}(z_0)\big),V\big(\phi^{s}_{\mathbf{X}}(z_0)\big)\bigg)+ D_{\phi^{s}_{\mathbf{X}}(z_0)}V \bigg(D_{\phi^{s}_{\mathbf{X}}(z_0)}V \big(V(\phi^{s}_{\mathbf{X}}(z_0))\big)\bigg).
		\end{split}
	\end{align}
	Also,
	\begin{align}\label{expa_2}
		\begin{split}
			V(\phi_{\mathbf{X}}(z_0))^{\#}_{s,t} &= D_{\phi^{s}_{\mathbf{X}}(z_0)}V\big(\phi_{\mathbf{X}}(z_0)^{\#}_{s,t}\big) \\
			&\quad + \frac{1}{2}D^{2}_{\phi^{s}_{\mathbf{X}}(z_0)}V\bigg(V\big(\phi^{s}_{\mathbf{X}}(z_0)\big) \delta X_{s,t},D_{\phi^{s}_{\mathbf{X}}(z_0)} V \big( V(\phi^{s}_{\mathbf{X}}(z_0))\big)\mathbb{X}_{s,t}^2 + \phi_{\mathbf{X}}(z_0)^{\#}_{s,t}\bigg)\\ 
			&\quad + \frac{1}{2}D^{2}_{\phi^{s}_{\mathbf{X}}(z_0)}V\bigg(D_{\phi^{s}_{\mathbf{X}}(z_0)} V \big( V(\phi^{s}_{\mathbf{X}}(z_0))\big)\mathbb{X}_{s,t}^2 + \phi_{\mathbf{X}}(z_0)^{\#}_{s,t},(\delta\phi_{\mathbf{X}}(z_0))_{s,t}\bigg) \\
			&\quad + \int_{0}^{1}\frac{(1-\sigma)^2}{2}D^{3}_{\sigma \phi^{t}_{\mathbf{X}}(z_0)+(1-\sigma)\phi^{s}_{\mathbf{X}}(z_0)}V\bigg((\delta\phi_{\mathbf{X}}(z_0))_{s,t},(\delta\phi_{\mathbf{X}}(z_0))_{s,t},(\delta\phi_{\mathbf{X}}(z_0))_{s,t}\bigg) \, \mathrm{d}\sigma,\\
			\big(V(\phi_{\mathbf{X}}(z_0))^{(1)}\big)^{\#}_{s,t} &= D_{\phi^{s}_{\mathbf{X}}(z_0)}^{2} V\bigg(  D_{\phi^{s}_{\mathbf{X}}(z_0)}V( V(\phi^{s}_{\mathbf{X}}(z_0)))\mathbb{X}_{s,t}^2+\phi_{\mathbf{X}}(z_0)^{\#}_{s,t} ,V(\phi^{s}_{\mathbf{X}}(z_0))\bigg) \\
			&\quad + D_{\phi^{s}_{\mathbf{X}}(z_0)}V\big((\phi_{\mathbf{X}}(z_0)^{(1)})^{\#}_{s,t}\big) \\
			&\quad + \int_{0}^{1}(1-\sigma)D^{3}_{\sigma\phi^{t}_{\mathbf{X}}(z_0)+(1-\sigma)\phi^{s}_{\mathbf{X}}(z_0)}V\big((\delta\phi_{\mathbf{X}}(z_0))_{s,t},(\delta\phi_{\mathbf{X}}(z_0))_{s,t},V(\phi^{s}_{\mathbf{X}}(z_0))\big) \, \mathrm{d}\sigma \\
			&\quad + (D_{\phi^{t}_{\mathbf{X}}(z_0)}V-D_{\phi^{s}_{\mathbf{X}}(z_0)}V)((\delta V(\phi_{\mathbf{X}}(z_0)))_{s,t}) .
		\end{split}
	\end{align}
	From Lemma \ref{expansion_2}, the tuple \( \big(V(\phi^{s}_{\mathbf{X}}(z_0)), V(\phi^{s}_{\mathbf{X}}(z_0))^{(1)}, V(\phi^{s}_{\mathbf{X}}(z_0))^{(2)}\big) \) is controlled by \( \mathbf{X} \), i.e.,
	\begin{align*}
		&(\delta V(\phi_{\mathbf{X}}(z_0)))_{s,t}=V(\phi^{s}_{\mathbf{X}}(z_0))^{(1)} \delta X_{s,t}+V(\phi^{s}_{\mathbf{X}}(z_0))^{(2)}\mathbb{X}^{2}_{s,t}+	V(\phi_{\mathbf{X}}(z_0))^{\#}_{s,t}\\
		&\quad \big(\delta V(\phi_{\mathbf{X}}(z_0))^{(1)}\big)_{s,t}=V(\phi^{s}_{\mathbf{X}}(z_0))^{(2)}\delta X_{s,t}+	\big(V(\phi_{\mathbf{X}}(z_0))^{(1)}\big)^{\#}_{s,t}.
	\end{align*}
	Note that,
	\begin{align*}
		&\phi^{t}_{\mathbf{X}}(z_0)-\phi^{s}_{\mathbf{X}}(z_0) \\
		=\ &\int_{s}^{t}V(\phi^{\tau}_{\mathbf{X}}(z_0))\mathrm{d}\mathbf{X}_\tau - V(\phi^{s}_{\mathbf{X}}(z_0)) X_{s,t}-V(\phi^{s}_{\mathbf{X}}(z_0))^{(1)}\mathbb{X}^{2}_{s,t} -V(\phi^{s}_{\mathbf{X}}(z_0))^{(2)}\mathbb{X}^{3}_{s,t} \\
		&\quad + \int_{s}^{t}V_{0}(\phi^{\tau}_{\mathbf{X}}(z_0)) \, \mathrm{d}\tau + V(\phi^{s}_{\mathbf{X}}(z_0)) X_{s,t} + V(\phi^{s}_{\mathbf{X}}(z_0))^{(1)}\mathbb{X}^{2}_{s,t}+V(\phi^{s}_{\mathbf{X}}(z_0))^{(2)}\mathbb{X}^{3}_{s,t}.
	\end{align*}
	From \eqref{SSWWEE},
	\begin{align}\label{Sewing_11}
		\begin{split}
			&\left\| \int_{s}^{t}V(\phi^{\tau}_{\mathbf{X}}(z_0))\mathrm{d}\mathbf{X}_\tau - V(\phi^{s}_{\mathbf{X}}(z_0)) \delta X_{s,t}-V(\phi^{s}_{\mathbf{X}}(z_0))^{(1)}\mathbb{X}^{2}_{s,t} -V(\phi^{s}_{\mathbf{X}}(z_0))^{(2)}\mathbb{X}^{3}_{s,t}\right\| \\ 
			\leq\  &C_2\bigg[ \big\Vert \big(V(\phi_{\mathbf{X}}(z_0))\big)^{\#}\big\Vert_{3\gamma,[s,t]}\Vert\mathbf{X}\Vert_{\gamma,[s,t]} + \big\Vert \big(V(\phi_{\mathbf{X}}(z_0))^{(1)}\big)^{\#}\big\Vert_{2\gamma,[s,t]}\Vert\mathbf{X}\Vert_{\gamma,[s,t]}^2 \\
			&\qquad + \big\Vert V(\phi^{.}_{\mathbf{X}}(z_0))^{(2)}\big\Vert_{\gamma,[s,t]}\Vert\mathbf{X}\Vert_{\gamma,[s,t]}^3 \bigg](t-s)^{4\gamma}.
		\end{split}
	\end{align}
	Moreover,
	\begin{align}\label{sharp}
		\big(\phi_{\mathbf{X}}(z_0)\big)^{\#}_{s,t}=\int_{s}^{t}V(\phi^{\tau}_{\mathbf{X}}(z_0))\mathrm{d}\mathbf{X}_\tau-V(\phi^{s}_{\mathbf{X}}(z_0)) \delta X_{s,t}-	V(\phi^{s}_{\mathbf{X}}(z_0))^{(1)} \mathbb{X}^{2}_{s,t}+\int_{s}^{t}V_{0}(\phi^{\tau}_{\mathbf{X}}(z_0)) \, \mathrm{d}\tau.
	\end{align}
	Assume that \([s,t] \subseteq [0,T]\). From \eqref{expa_1}, \eqref{expa_2}, and the assumptions on \(V\), we can find a constant \(M_1\) such that
	\begin{align}\label{(1)}
		\begin{split}
			&\big\Vert \big(V(\phi_{\mathbf{X}}(z_0))\big)^{\#}\big\Vert_{3\gamma,[s,t]}\\&\quad\leq M_1 \bigg[\big(1+\Vert\mathbf{X}\Vert_{\gamma,[s,t]} + 
			\Vert \phi_{\mathbf{X}}(z_0)\Vert_{\gamma,[s,t]}\big)\Vert\phi_{\mathbf{X}}(z_0)^{\#}\Vert_{3\gamma,[s,t]}\\&\quad+\big( \Vert \phi_{\mathbf{X}}(z_0)\Vert_{\gamma,[s,t]}^2+\Vert\mathbf{X}\Vert_{\gamma,[s,t]}^2\big)\Vert \phi_{\mathbf{X}}(z_0)\Vert_{\gamma,[s,t]}+\Vert\mathbf{X}\Vert_{\gamma,[s,t]}^3\bigg].
		\end{split}
	\end{align}
	Also from  \eqref{VVVZZZ}, \eqref{expa_1} and \eqref{expa_2}
	\begin{align}\label{(2)}
		\begin{split}
			\big\Vert \big(V(\phi_{\mathbf{X}}(z_0))^{(1)}\big)^{\#}\big\Vert_{2\gamma,[s,t]}&\leq M_1 \big[\Vert\phi_{\mathbf{X}}(z_0)^{\#}\Vert_{3\gamma,[s,t]}+\Vert \phi_{\mathbf{X}}(z_0)\Vert_{\gamma,[s,t]}^2+\Vert\mathbf{X}\Vert_{\gamma,[s,t]}^2\big],
			\\ \big\Vert V(\phi^{.}_{\mathbf{X}}(z_0))^{(2)}\big\Vert_{\gamma,[s,t]}&\leq M_{1}\Vert \phi_{\mathbf{X}}(z_0)\Vert_{\gamma,[s,t]}.
		\end{split}
	\end{align}
	The goal is to use inequality \eqref{Sewing_11} to obtain a priori bounds for \( \big\Vert \big(\phi_{\mathbf{X}}(z_0)\big)^{\#} \big\Vert_{3\gamma,[s,t]} \) on an arbitrary interval \( [s,t] \). The idea is to apply the bounds obtained in \eqref{(1)} and \eqref{(2)}, and substitute them into the right-hand side of \eqref{Sewing_11}. Assume that \( [s,t] \subseteq [0,T] \) and \( t - s \leq 1 \). Recall \eqref{sup_norm} and the assumption on \( V_0 \) stated in \eqref{BNM<<<}. Since \( (t - s)^{\gamma} \leq (t - s)^{1 - 3\gamma} \), after substituting the bounds, it follows that there exist two increasing polynomials \( Q_1 \) and \( Q_2 \) (in both variables) such that
	\begin{align}\label{ASDC}
		\begin{split}
			&\big\Vert \big(\phi_{\mathbf{X}}(z_0)\big)^{\#}\big\Vert_{3\gamma,[s,t]}\leq\\&(t-s)^{1-3\gamma}Q_{1}(\Vert\mathbf{X}\Vert_{\gamma,[s,t]} , 
			\Vert \phi_{\mathbf{X}}(z_0)\Vert_{C^{\gamma},[s,t]})\Vert\phi_{\mathbf{X}}(z_0)^{\#}\Vert_{3\gamma,[s,t]}+Q_2(\Vert\mathbf{X}\Vert_{\gamma,[s,t]} , 
			\Vert \phi_{\mathbf{X}}(z_0)\Vert_{C^{\gamma},[s,t]}).
		\end{split}
	\end{align}
	Let \( \epsilon \in (0,1) \) and 
	\begin{align}\label{PPP}
		\tau^{1-3\gamma} Q_1\big(\Vert\mathbf{X}\Vert_{\gamma,[0,T]},\Vert \phi_{\mathbf{X}}(z_0)\Vert_{C^{\gamma},[0,T]}\big)
		= 1-\epsilon.
	\end{align}
	We choose a finite sequence \( (\tau_n)_{0 \leq n \leq N(\epsilon,\mathbf{X},z_0)} \) in \( [0,T] \), such that \( \tau_0 = 0 \), \( \tau_{N(\epsilon,\mathbf{X},z_0)} = T \), and for all \( 0 \leq n < N(\epsilon,\mathbf{X},z_0) - 1 \), we have \( \tau_{n+1} - \tau_n = \tau \).
	Then, it follows from \eqref{ASDC} that, for \( M_{\epsilon} = \frac{1}{\epsilon} \) and all \( 0 \leq n < N(\epsilon, \mathbf{X}, z_0) - 1 \),
	\begin{align}\label{NBBBV}
		\begin{split}
			&\big\Vert \big(\phi_{\mathbf{X}}(z_0)\big)^{\#}\big\Vert_{3\gamma,[\tau_n,\tau_{n+1}]}\leq M_{\epsilon}Q_2\big(\Vert\mathbf{X}\Vert_{\gamma,[0,T]},\Vert \phi_{\mathbf{X}}(z_0)\Vert_{C^{\gamma},[0,T]}\big).
		\end{split}
	\end{align}
	Note that from \eqref{PPP},
	\begin{align}\label{CFCF}
		N(\epsilon,\mathbf{X},z_0) =\lfloor\frac{T}{\tau}\rfloor +1= \big\lfloor T M_{\epsilon}^{\frac{1}{1-3\gamma}}Q_1(\Vert\mathbf{X}\Vert_{\gamma,[0,T]},\Vert \phi_{\mathbf{X}}(z_0)\Vert_{C^{\gamma},[0,T]})^{\frac{1}{1-3\gamma}} \big\rfloor +1 .
	\end{align}
	Assume $\tau<\nu<\upsilon$, then from \eqref{sharp},
	\begin{align}\label{(4)}
		\phi_{\mathbf{X}}(z_0) ^{\#}_{\tau,\upsilon}=\phi_{\mathbf{X}}(z_0) ^{\#}_{\tau,\nu}+\phi_{\mathbf{X}}(z_0)^{\#}_{\nu,\upsilon}+\big(\phi_{\mathbf{X}}(z_0)^{(1)}\big)^{\#}_{\tau,\nu} \delta X_{\nu,\upsilon}+\delta\big(D_{\phi_{\mathbf{X}}(z_0)}V\big(V(\phi_{\mathbf{X}}(z_0))\big)\big)_{\tau,\nu}\mathbb{X}_{\nu,\upsilon}^2.
	\end{align}		
	Therefore, for a constant \( M_2 > 0 \), we obtain
	\begin{align*}
		&\Vert\phi_{\mathbf{X}}(z_0) ^{\#}\Vert_{3\gamma,[\tau,\upsilon]}\leq \Vert\phi_{\mathbf{X}}(z_0)^{\#}\Vert_{3\gamma,[\tau,\nu]}+\Vert\phi_{\mathbf{X}}(z_0)^{\#}\Vert_{3\gamma,[\nu,\upsilon]}\\ &\quad+\Vert(\phi_{\mathbf{X}}(z_0)^{(1)})^{\#}\Vert_{2\gamma,[\tau,\nu]}\Vert\mathbf{X}\Vert_{\gamma,[0,T]}+M_2\Vert\phi_{\mathbf{X}}(z_0)\Vert_{\gamma,[0,T]}\Vert\mathbf{X}\Vert_{\gamma,[0,T]}^2.
	\end{align*}
	From \eqref{VVVZZZ}, we can find a constant $M_3$ such that for every $[\tau,\nu]\subseteq [0,T]$,
	\begin{align}\label{(7)}
		\begin{split}
			&	\big\Vert\big(\phi_{\mathbf{X}}(z_0)^{(1)}\big)^{\#}\big\Vert_{2\gamma,[\tau,\nu]}\leq M_3\bigg(\big\Vert\big(\phi_{\mathbf{X}}(z_0)\big)^{\#}\big\Vert_{3\gamma,[\tau,\nu]}+\Vert\mathbf{X}\Vert^{2}_{\gamma,[0,T]}+\Vert \phi_{\mathbf{X}}(z_0)\Vert_{\gamma,[0,T]}^2\bigg).
		\end{split}
	\end{align}
	From and \eqref{(4)} and \eqref{(7)},
	\begin{align}\label{(8)}
		\begin{split}
			&\Vert\phi_{\mathbf{X}}(z_0) ^{\#}\Vert_{3\gamma,[\tau,\upsilon]}\leq (1+M_3\Vert\mathbf{X}\Vert_{\gamma,[0,T]})\Vert\phi_{\mathbf{X}}(z_0)^{\#}\Vert_{3\gamma,[\tau,\nu]}+\Vert\phi_{\mathbf{X}}(z_0)^{\#}\Vert_{3\gamma,[\nu,\upsilon]}\\ &\quad+ M_4\Vert\mathbf{X}\Vert_{\gamma,[0,T]}\big(\Vert\mathbf{X}\Vert^{2}_{\gamma,[0,T]}+\Vert \phi_{\mathbf{X}}(z_0)\Vert_{\gamma,[0,T]}^2\big),
		\end{split}
	\end{align}
	where $M_4$ is a constant.
	Thus, from \eqref{NBBBV} and $0\leq n<N(\epsilon,\mathbf{X},z_0)-1$,

	\begin{align}\label{BVVVCCZZS}
		\begin{split}
			&\Vert\phi_{\mathbf{X}}(z_0) ^{\#}\Vert_{3\gamma,[\tau_{n},T]}\leq (1+M_3\Vert\mathbf{X}\Vert_{\gamma,[0,T]})\Vert\phi_{\mathbf{X}}(z_0)^{\#}\Vert_{3\gamma,[\tau_{n},\tau_{n+1}]}+\Vert\phi_{\mathbf{X}}(z_0) ^{\#}\Vert_{3\gamma,[\tau_{n+1},T]}\\
			&\quad+M_4\Vert\mathbf{X}\Vert_{\gamma,[0,T]}\big(\Vert\mathbf{X}\Vert^{2}_{\gamma,[0,T]}+\Vert \phi_{\mathbf{X}}(z_0)\Vert_{\gamma,[0,T]}^2\big)\\
			&\qquad \leq M_{\epsilon}(1+M_3\Vert\mathbf{X}\Vert_{\gamma,[0,T]})Q_2(\Vert\mathbf{X}\Vert_{\gamma,[0,T]},\Vert \phi_{\mathbf{X}}(z_0)\Vert_{C^{\gamma},[0,T]})\\&\myquad[3] + M_4\Vert\mathbf{X}\Vert_{\gamma,[0,T]}\big(\Vert\mathbf{X}\Vert^{2}_{\gamma,[0,T]}+\Vert \phi_{\mathbf{X}}(z_0)\Vert_{\gamma,[0,T]}^2\big)+\Vert\phi_{\mathbf{X}}(z_0) ^{\#}\Vert_{3\gamma,[\tau_{n+1},T]}\\
			&\myquad[4]=Q_{4}(\Vert\mathbf{X}\Vert_{\gamma,[0,T]},\Vert \phi_{\mathbf{X}}(z_0)\Vert_{C^{\gamma},[0,T]})+\Vert\phi_{\mathbf{X}}(z_0) ^{\#}\Vert_{3\gamma,[\tau_{n+1},T]}.
		\end{split}
	\end{align}
	So, if we start with $\tau_0=0$ and recursively apply on \eqref{BVVVCCZZS},
	\begin{align}\label{M<OO}
		\Vert\phi_{\mathbf{X}}(z_0) ^{\#}\Vert_{3\gamma,[0,T]}\leq (	N(\epsilon,\mathbf{X},z_0)+1)Q_{4}(\Vert\mathbf{X}\Vert_{\gamma,[0,T]},\Vert \phi_{\mathbf{X}}(z_0)\Vert_{C^{\gamma},[0,T]}).
	\end{align}
	From \eqref{sup_norm} and \eqref{CFCF}, we can observe that \( N(\epsilon, \mathbf{X}, z_0) \) exhibits polynomial growth in terms of \( \Vert \mathbf{X} \Vert_{\gamma, [0,T]} \) and \( \Vert z_0 \Vert \). Therefore, our claim follows from \eqref{sup_norm} and \eqref{M<OO}.
\end{proof}
	In the previous proposition, we obtained an a priori bound for the solution map in the \( \mathcal{D}^{\gamma}_{\mathbf{X}} \)-norm, which was shown to be polynomial in both the norm of the initial value and \( \Vert \mathbf{X} \Vert_{\gamma,[0,T]} \). As we will see in the following result, this estimate is crucial: it enables us to establish an a priori bound for the linearized equation. Such a bound is essential for the application of the multiplicative ergodic theorem.
\begin{proposition}\label{DDSD}
	Let there exist an increasing polynomial \( P_1 \) such that
	\begin{align}\label{41}
		\forall z\in\mathbb{R}^m:\ \  \Vert D_{z}V_{0}\Vert\leq P_{1}(\Vert z\Vert).
	\end{align}
	Then the solution to equation~\eqref{SDE} is differentiable, and there exists a polynomial \( Q_1 \) such that
	\begin{align}\label{PPOL}
		\vertiii{D_{z_0}\phi_{\mathbf{X}}[\Bar{z}]}_{\mathcal{D}^{\gamma}_{\mathbf{X}}([0,T])}\leq \Vert\Bar{z}\Vert \exp\big(Q_{1}(\Vert z_0\Vert,\Vert\mathbf{X}\Vert_{\gamma,[0,T]})\big)
	\end{align}
	for every \( \bar{z}, z_0 \in \mathbb{R}^m \) and \( \gamma \neq \frac{1}{3} \).
\end{proposition}
\begin{proof}
	Note that by our regularity assumption on $V_0$ and $V$, $\phi_{\mathbf{X}}^{t}(.)$ is differentiable with respect to initial values. 
	To show this claim, let $\delta: \mathbb{R}^m\rightarrow \mathbb{R}$, be a $C^{\infty}-$ function such that $\text{support}(\delta)\subset B(0,2)$ and $\delta|_{B(0,1)}=1$. For $R>0$, set $\overline{V}_0(z):=\delta (\frac{z}{R})V_{0}(z)$. Consider the following equation
	\begin{align}\label{CXC}
		\mathrm{d}Z_{t} = V(Z_t) \, \mathrm{d} \mathbf{X}_t+\overline{V}_0(Z_t)\, \mathrm{d}t, \ \ \  Z_{0}=z_0\in \mathbb{R}^m.
	\end{align}
	By construction, \( \overline{V}_0 \in C^{1}_b \). Therefore, under the assumptions on \( V \), and by \cite[Theorem 11.6]{FV10}, the solutions of \eqref{CXC} are differentiable with respect to the initial condition.
	Note that the solutions to \eqref{CXC} coincide locally in time with those of \eqref{SDE} in a neighborhood of \( z \). Since \( z \) and \( R \) are arbitrary, we can conclude that the solutions to \eqref{SDE} are also differentiable with respect to the initial condition. Moreover, the derivative satisfies the following equation:
	\begin{align}\label{NMNMNAA}
		\mathrm{d}D_{z_0}\phi^{t}_{\mathbf{X}}[\Bar{z}]=D_{\phi^{t}_{\mathbf{X}}(z_0)}V (D_{z_0}\phi^{t}_{\mathbf{X}}[\Bar{z}])\mathrm{d}\mathbf{X}_t+D_{\phi^{t}_{\mathbf{X}}(z_0)}V_0 (D_{z_0}\phi^{t}_{\mathbf{X}}[\Bar{z}])\mathrm{d}t, \ \ \ D_{z_0}\phi^{0}_{\mathbf{X}}[\Bar{z}]=\Bar{z}\in \mathbb{R}^m.
	\end{align}
	We use \eqref{Linearized} to obtain the priori bound \eqref{PPOL}. Note that
	\begin{align}\label{Linearized}
		\begin{split}
			&D_{z_0}\phi^{t}_{\mathbf{X}}[\Bar{z}]-D_{z_0}\phi^{s}_{\mathbf{X}}[\Bar{z}] \\
			=\ &\int_{s}^{t}D_{\phi^{\tau}_{\mathbf{X}}(z_0)}V( D_{z_0}\phi^{\tau}_{\mathbf{X}}[\Bar{z}])\mathrm{d}\mathbf{X}_\tau -D_{\phi^{s}_{\mathbf{X}}(z_0)}V (D_{z_0}\phi^{s}_{\mathbf{X}}[\Bar{z}])\delta X_{s,t}- \big(D_{\phi^{s}_{\mathbf{X}}(z_0)}V(D_{z_0}\phi^{s}_{\mathbf{X}}[\Bar{z}])\big)^{(1)}\mathbb{X}^{2}_{s,t} \\
			&\quad -  \big(D_{\phi^{s}_{\mathbf{X}}(z_0)}V(D_{z_0}\phi^{s}_{\mathbf{X}}[\Bar{z}])\big)^{(2)}\mathbb{X}^{3}_{s,t} + \int_{s}^{t} D_{\phi^{\tau}_{\mathbf{X}}(z_0)}V_0( D_{z_0}\phi^{\tau}_{\mathbf{X}}[\Bar{z}]) \, \mathrm{d}\tau+
			D_{\phi^{s}_{\mathbf{X}}(z_0)}V (D_{z_0}\phi^{s}_{\mathbf{X}}[\Bar{z}])\delta X_{s,t} \\
			&\quad +\big(D_{\phi^{s}_{\mathbf{X}}(z_0)}V(D_{z_0}\phi^{s}_{\mathbf{X}}[\Bar{z}])\big)^{(1)}\mathbb{X}^{2}_{s,t}+ \big(D_{\phi^{s}_{\mathbf{X}}(z_0)}V(D_{z_0}\phi^{s}_{\mathbf{X}}[\Bar{z}])\big)^{(2)}\mathbb{X}^{3}_{s,t}.
		\end{split}
	\end{align}
	Since \( D_{z_0} \phi_{\mathbf{X}}[\bar{z}] \) solves \eqref{NMNMNAA}, we can calculate the terms of the expansion of \( D_{z_0} \phi_{\mathbf{X}}[\bar{z}] \) as given in \eqref{expansion}. From Lemma~\ref{expansion_2}, we obtain:
	
	\begin{align}\label{A_1}
		\begin{split}
			(D_{z_0}\phi^{s}_{\mathbf{X}}[\Bar{z}])^{(1)} &= D_{\phi^{s}_{\mathbf{X}}(z_0)}V(D_{z_0}\phi^{s}_{\mathbf{X}}[\Bar{z}]), \\
			(D_{z_0}\phi^{s}_{\mathbf{X}}[\Bar{z}])^{(2)} &=\big(D_{\phi^{s}_{\mathbf{X}}(z_0)}V(D_{z_0}\phi^{s}_{\mathbf{X}}[\Bar{z}])\big)^{(1)} \\&=D^{2}_{\phi^{s}_{\mathbf{X}}(z_0)}V\big(V\big(\phi^{s}_{\mathbf{X}}(z_0)\big),D_{z_0}\phi^{s}_{\mathbf{X}}[\Bar{z}]\big)+ D_{\phi^{s}_{\mathbf{X}}(z_0)}V \big(D_{\phi^{s}_{\mathbf{X}}(z_0)}V( D_{z_0}\phi^{s}_{\mathbf{X}}[\Bar{z}])\big), \\
			(D_{z_0}\phi_{\mathbf{X}}[\Bar{z}])^{\#}_{s,t} &= (\delta D_{z_0}\phi_{\mathbf{X}}[\Bar{z}])_{s,t}-D_{\phi^{s}_{\mathbf{X}}(z_0)}V(D_{z_0}\phi^{s}_{\mathbf{X}}[\Bar{z}])\delta X_{s,t}- \big(D_{\phi^{s}_{\mathbf{X}}(z_0)}V(D_{z_0}\phi^{s}_{\mathbf{X}}[\Bar{z}])\big)^{(1)}\mathbb{X}_{s,t}^{2} \quad \text{and} \\
			\big((D_{z_0}\phi_{\mathbf{X}}[\Bar{z}])^{1}\big)^{\#}_{s,t} &= \big(\delta D_{\phi_{\mathbf{X}}(z_0)}V( D_{z_0}\phi_{\mathbf{X}}[\Bar{z}])\big)_{s,t}-\big(D_{\phi^{s}_{\mathbf{X}}(z_0)}V(D_{z_0}\phi^{s}_{\mathbf{X}}[\Bar{z}])\big)^{(1)}\delta{X}_{s,t}.
		\end{split}
	\end{align}
	By virtue of \eqref{composite}, we have
	\begin{align}\label{CXX}
		\begin{split}
			&\vertiii{D_{\phi^{\tau}_{\mathbf{X}}(z_0)}V}_{\mathcal{D}^{\gamma}_{\mathbf{X}}([s,t])} \lesssim \big(1+ \vertiii{\phi_{\mathbf{X}}(z_0)}_{\mathcal{D}^{\gamma}_{\mathbf{X}}([u,v])}^9+\Vert\mathbf{X}\Vert_{\gamma,[s,t]}^9\big). \\
		\end{split}
	\end{align}
	By \eqref{CDCa},
	\begin{align}\label{NM855}
		\begin{split}
			&\vertiii{D_{\phi_{\mathbf{X}}(z_0)}V\big(D_{z_0}\phi_{\mathbf{X}}[\Bar{z}]\big)}_{\mathcal{D}^{\gamma}_{\mathbf{X}}([s,t])}\\&\quad\lesssim (1+\Vert\mathbf{X}\Vert_{\gamma,[s,t]}^4)\big(1+ \vertiii{\phi_{\mathbf{X}}(z_0)}_{\mathcal{D}^{\gamma}_{\mathbf{X}}([u,v])}^9+\Vert\mathbf{X}\Vert_{\gamma,[s,t]}^9\big)\vertiii{D_{z_0}\phi_{\mathbf{X}}[\Bar{z}]}_{\mathcal{D}^{\gamma}_{\mathbf{X}}([s,t])}.
		\end{split}
	\end{align}
	Assume that \( s, t \in [0, T] \) with \( s \leq t \). Then, by \eqref{NMUIO}, \eqref{A_1}, and \eqref{NM855}, we have
	\begin{align}\label{Sewing lemma}
		\begin{split}
			&\vertiii{\int_{s}D_{\phi^{\tau}_{\mathbf{X}}(z_0)}V\big(D_{z_0}\phi^{\tau}_{\mathbf{X}}[\Bar{z}]\big)\mathrm{d}\mathbf{X}_\tau}_{\mathcal{D}^{\gamma}_{\mathbf{X}}([s,t])}
			\lesssim (1+\Vert\mathbf{X}\Vert_{\gamma,[s,t]}^3)\Vert D_{z_0}\phi^{s}_{\mathbf{X}}[\Bar{z}]\Vert\\&+(t-s)^\gamma\underbrace{(1+\Vert\mathbf{X}\Vert_{\gamma,[s,t]}^4)^2\big(1+ \vertiii{\phi_{\mathbf{X}}(z_0)}_{\mathcal{D}^{\gamma}_{\mathbf{X}}([s,t])}^9+\Vert\mathbf{X}\Vert_{\gamma,[s,t]}^9\big)}_{Q_{5}( \Vert\mathbf{X}\Vert_{\gamma,[s,t]},\vertiii{\phi_{\mathbf{X}}(z_0)}_{\mathcal{D}^{\gamma}_{\mathbf{X}}([s,t])}\Vert)}\vertiii{D_{z_0}\phi_{\mathbf{X}}[\Bar{z}]}_{\mathcal{D}^{\gamma}_{\mathbf{X}}([s,t])}.
		\end{split}
	\end{align}
	From \eqref{41},
	\begin{align}\label{DDrift}
		\begin{split}
			&\vertiii{\big(\int_{s} D_{\phi^{\tau}_{\mathbf{X}}(z_0)}V_0( D_{z_0}\phi^{\tau}_{\mathbf{X}}[\Bar{z}]) \, \mathrm{d}\tau,0,0\big)}_{\mathcal{D}^{\gamma}_{\mathbf{X}}([s,t])}\\
			&\lesssim (t-s)^{1-3\gamma}P_{1}(\vertiii{\phi_{\mathbf{X}}(z_0)}_{\mathcal{D}^{\gamma}_{\mathbf{X}}([0,T])}) \vertiii{D_{z_0}\phi^{.}_{\mathbf{X}}[\Bar{z}]}_{\mathcal{D}^{\gamma}_{\mathbf{X}}([s,t])}.
		\end{split}
	\end{align}
	Let \( [s, t] \) be an arbitrary subinterval of \( [0, T] \) such that \( t - s \leq 1 \). Then, since \( (t - s)^{\gamma} \leq (t - s)^{1 - 3\gamma} \), it follows from \eqref{Sewing lemma} and \eqref{DDrift} that there exists a constant \( M > 1 \) such that
	\begin{align*}
		&\quad	\vertiii{D_{z_0}\phi_{\mathbf{X}}[\Bar{z}]}_{\mathcal{D}^{\gamma}_{\mathbf{X}}([s,t])}\leq M\bigg((1+\Vert\mathbf{X}\Vert_{\gamma,[0,T]}^3)\Vert D_{z_0}\phi^{s}_{\mathbf{X}}[\Bar{z}]\Vert\\&+(t-s)^{1-3\gamma}\big(Q_{5}( \Vert\mathbf{X}\Vert_{\gamma,[0,T]},\vertiii{\phi_{\mathbf{X}}(z_0)}_{\mathcal{D}^{\gamma}_{\mathbf{X}}([0,T])}\Vert)+P_{1}(\vertiii{\phi_{\mathbf{X}}(z_0)}_{\gamma,[0,T]})\big)\vertiii{D_{z_0}\phi_{\mathbf{X}}[\Bar{z}]}_{\mathcal{D}^{\gamma}_{\mathbf{X}}([s,t])}\bigg).
	\end{align*}

	Let $0<\epsilon<1$. We choose a finite sequence $(\tau_n)_{0\leq n\leq \tilde{N}(\epsilon,\mathbf{X},z_0)}$ in $[0,T]$, such that $\tau_0=0$ and $\tau_{\tilde{N}(\epsilon,\mathbf{X},z_0)}=T$. For $0\leq n<\tilde{N}(\epsilon,\mathbf{X},z_0)-1$, we assume $\tau_{n+1}-\tau_{n}=\tau$ and
	\begin{align*}
		\tau^{1-3\gamma}M\big(Q_{5}( \Vert\mathbf{X}\Vert_{\gamma,[0,T]},\vertiii{\phi_{\mathbf{X}}(z_0)}_{\mathcal{D}^{\gamma}_{\mathbf{X}}([0,T])}\Vert)+P_{1}(\vertiii{\phi_{\mathbf{X}}(z_0)}_{\mathcal{D}^{\gamma}_{\mathbf{X}}([0,T])})\big)=1-\epsilon.
	\end{align*}
	From \eqref{Sewing lemma}, \eqref{DDrift} and the choice of $\tau$, for $M_\epsilon=\frac{M}{\epsilon}$
	\begin{align}\label{greeding}
		\begin{split}
			&	0\leq n<\tilde{N}(\epsilon,\mathbf{X},z_0)-1:\ \ \\&\quad		
			\vertiii{D_{z_0}\phi_{\mathbf{X}}[\Bar{z}]}_{\mathcal{D}^{\gamma}_{\mathbf{X}}([\tau_n,\tau_{n+1}])}\leq M_{\epsilon}(1+\Vert\mathbf{X}\Vert_{\gamma,[0,T]}^3)\Vert D_{z_0}\phi^{\tau_n}_{\mathbf{X}}[\Bar{z}]\Vert.
		\end{split}
	\end{align} 
	Note that
	\begin{align}\label{BVBVA}
		\begin{split}
			&\tilde{N}(\epsilon,\mathbf{X},z_0) =\lfloor\frac{T}{\tau}\rfloor +1 \\&\quad=\bigg\lfloor T \big(M_\epsilon\big(Q_{5}( \Vert\mathbf{X}\Vert_{\gamma,[0,T]},\vertiii{\phi_{\mathbf{X}}(z_0)}_{\mathcal{D}^{\gamma}_{\mathbf{X}}([0,T])}\Vert)+P_{1}(\vertiii{\phi_{\mathbf{X}}(z_0)}_{\mathcal{D}^{\gamma}_{\mathbf{X}}([0,T])})\big) \big)^{\frac{1}{ 1-3\gamma}} \bigg\rfloor +1.
		\end{split}
	\end{align}
	Since \eqref{greeding} holds for every $0\leq n<\tilde{N}(\epsilon,\mathbf{X},z_0)-1$ and $M_\epsilon>1$, we conclude
	\begin{align}\label{M<LOPLKKJa}
		\sup_{0\leq n<\tilde{N}(\epsilon,\mathbf{X},z_0)-1}\vertiii{D_{z_0}\phi^{.}_{\mathbf{X}}[\Bar{z}]}_{\mathcal{D}^{\gamma}_{\mathbf{X}}([\tau_n,\tau_{n+1}])}\leq  \big(M_{\epsilon}(1+\Vert\mathbf{X}\Vert_{\gamma,[0,T]}^3)\big)^{\tilde{N}(\epsilon,\mathbf{X},z_0)+1}\Vert\Bar{z} \Vert.
	\end{align}
	From \eqref{DDFFD}, for every $0\leq n<\tilde{N}(\epsilon,\mathbf{X},z_0)-1$
	\begin{align*}
		&\vertiii{D_{z_0}\phi_{\mathbf{X}}[\Bar{z}]}_{\mathcal{D}^{\gamma}_{\mathbf{X}}([\tau_n,T]}\leq\\ \quad& (\Vert\mathbf{X}\Vert_{\gamma,[0,T]}+1)^2\vertiii{D_{z_0}\phi_{\mathbf{X}}[\Bar{z}]}_{\mathcal{D}^{\gamma}_{\mathbf{X}}([\tau_n,\tau_{n+1}]}+\vertiii{D_{z_0}\phi_{\mathbf{X}}[\Bar{z}]}_{\mathcal{D}^{\gamma}_{\mathbf{X}}([\tau_{n+1},T]}.
	\end{align*}
	Therefore from \eqref{M<LOPLKKJa}
	\begin{align}\label{BVZQWE}
		\vertiii{D_{z_0}\phi_{\mathbf{X}}[\Bar{z}]}_{\mathcal{D}^{\gamma}_{\mathbf{X}}([0,T]}\leq (\tilde{N}(\epsilon,\mathbf{X},z_0)+1)(\Vert\mathbf{X}\Vert_{\gamma,[0,T]}+1)^2 \big(M_{\epsilon}(1+\Vert\mathbf{X}\Vert_{\gamma,[0,T]}^3)\big)^{\tilde{N}(\epsilon,\mathbf{X},z_0)+1}\Vert\Bar{z} \Vert.
	\end{align}
	Recall from Proposition~\ref{RS17} that, for a polynomial \( Q \), we have		\begin{align*}
		&\vertiii{\phi_{\mathbf{X}}(z_0)}_{\mathcal{D}^{\gamma}_{\mathbf{X}}([0, T])} \leq  Q(\Vert z_0\Vert,\Vert\mathbf{X}\Vert_{\gamma,[0,T]})
	\end{align*}
	Thus, \( \tilde{N}(\epsilon, \mathbf{X}, z_0) \) has polynomial growth in terms of \( \| \mathbf{X} \|_{\gamma,[0,T]} \) and \( \| z_0 \| \). Hence, from \eqref{BVZQWE}, we can obtain \eqref{PPOL}.
\end{proof}
	\begin{remark}
		The {a priori} bound established in Proposition~\ref{DDSD} resembles the type of bound commonly obtained for linear ordinary differential equations. The principal technical distinction, of course, lies in the definition of the integrals. As anticipated, the norms of the rough drivers naturally appear as exponential terms in the resulting estimates.
		This result has several noteworthy implications. Most importantly, taking the logarithm of this bound effectively cancels the exponential term, resulting in a polynomial dependence on both the norm of the path and the initial value. In the Gaussian setting, this polynomial bound is integrable.
		This observation leads to two consequences. First, as previously mentioned, it facilitates the application of the multiplicative ergodic theorem. Second, it enables local estimates of the Lipschitz norm of the map \( z \mapsto \phi_{\mathbf{X}}(z) \). Although this norm depends on the path, it remains sufficiently controlled to deduce the existence of invariant manifolds.  
		All of these statements will be rigorously justified in the next section.
	\end{remark}
Proposition~\ref{DDSD} leads to the following result:	\begin{corollary}\label{difference}
	There exists a polynomial \( Q_2 \) such that
	\begin{align*}
		\vertiii{\phi_{\mathbf{X}}(z_1)-\phi_{\mathbf{X}}(z_0)}_{\mathcal{D}^{\gamma}_{\mathbf{X}}([0, T])}\leqslant \Vert z_1-z_0\Vert\exp\big(Q_2(\Vert z_0\Vert,\Vert z_1\Vert,\Vert\mathbf{X}\Vert_{\gamma,[0,T]})\big),
	\end{align*}
	for all \( z_0, z_1 \in \mathbb{R}^m \).
\end{corollary}	
\begin{proof}
	Note that
	\begin{align*}
		\phi_{\mathbf{X}}^{t}(z_1)-\phi_{\mathbf{X}}^{t}(z_0)=\int_{0}^{1}D_{z_0+\theta(z_1-z_0)}\phi^{t}_{\mathbf{X}}[z_1-z_0] \, \mathrm{d}\theta.
	\end{align*}
	Now, it suffices to apply Proposition~\ref{DDSD}.
\end{proof}
\begin{remark}
	Let $\big({\phi}_{\mathbf{X}}(s,u,z_0)\big)_{u\geq s}$ solve equation
	\begin{align*}
		\mathrm{d}{\phi}_{\mathbf{X}}(s,t,z_0) = V\big({\phi}_{\mathbf{X}}(s,t,z_0)\big) \,  \mathrm{d} \mathbf{X}_t+V_{0}({\phi}_{\mathbf{X}}(s,t,z_0))\, \mathrm{d}t, \ \ \  {\phi}_{\mathbf{X}}(s,s,z_0)=z_0\in \mathbb{R}^m.
	\end{align*}
	In particular, ${\phi}_{\mathbf{X}}(0,t,z_0)={\phi_{\mathbf{X}}^t(z_0)}$. Also for every $0 \leq 
	s <T$, the bound which we obtained in Proposition \ref{RS17} holds for $\vertiii{{\phi}_{\mathbf{X}}(s,\cdot,z_0)}_{\mathcal{D}^{\gamma}_{\mathbf{X}}([s, T])}$. In addition, we have the semiflow property, i.e.  
	\begin{align*}
		0\leq s\leq u\leq t:\ \     {\phi}_{\mathbf{X}}(s,t,z_0)={\phi}_{\mathbf{X}}(s,u,{\phi}_{\mathbf{X}}(u,t,z_0)).
	\end{align*}
	Assume further that \( \big(\psi_{\mathbf{X}}(s,u,z_0)\big)_{u \geq s} \) solves
	\begin{align*}
		&\mathrm{d}{\psi}_{\mathbf{X}}(s,t,z_0)[\bar{z}]\\&\quad=D_{{\phi}_{\mathbf{X}}(s,t,z_0)}V \big({\psi}_{\mathbf{X}}(s,t,z_0)[\bar{z}]\big) \, \mathrm{d}\mathbf{X}_t+D_{{\phi}_{\mathbf{X}}(s,t,z_0)}V_0 \big({\psi}_{\mathbf{X}}(s,t,z_0)[\bar{z}]\big)\, \mathrm{d}t, \ \ \ {\psi}_{\mathbf{X}}(s,s,z_0)[\bar{z}]=\Bar{z}\in \mathbb{R}^m.
	\end{align*}
	It is clear that \(\psi_{\mathbf{X}}(0,t,z_0) = D_{z_0} \phi_{\mathbf{X}}^{t}\), and the semiflow property holds for \(\psi_{\mathbf{X}}\).
\end{remark}
Thanks to the semiflow property, we can conclude that the map \( D_{z_0} \phi_{\mathbf{X}}^{t} : \mathbb{R}^m \to \mathbb{R}^m \) is invertible for every \( t \geq 0 \) and \( z_0 \in \mathbb{R}^m \). For future purposes, we also need to obtain a similar bound as in \eqref{PPOL} for
\[
\sup_{0 \leq t \leq T} \Vert (D_{z_0} \phi_{\mathbf{X}}^t)^{-1} \Vert_{\mathcal{L}(\mathbb{R}^m, \mathbb{R}^m)}.
\]
In the next lemma, we sketch how this bound can be obtained.	

\begin{lemma}\label{LLMMNNBB}
	Assume that the conditions of Proposition~\ref{DDSD} hold. Then, for the same polynomial \( Q_1 \),  
	\begin{align*}
		\sup_{0\leq t\leq T}\Vert{(D_{z_0}\phi_{\mathbf{X}}^{t}})^{-1}\Vert_{\mathcal{L}(\mathbb{R}^m,\mathbb{R}^m)}\leq \exp\big(Q_{1}(\Vert z_0\Vert,\Vert\mathbf{X}\Vert_{\gamma,[0,T]})\big).
	\end{align*}
\end{lemma}
	\begin{proof}
		Let us fix \( 0 < t_0 \leq T \) and define, for \( 0 \leq s, t \leq t_0 \),
		\begin{align*}
			\tilde{X}_{t} := X_{t_0 - t}, \quad 
			\tilde{\mathbb{X}}^2_{s,t} := -\mathbb{X}^2_{t_0 - t, t_0 - s}, \quad 
			\tilde{\mathbb{X}}^3_{s,t} := -\mathbb{X}^3_{t_0 - t, t_0 - s}.
		\end{align*}
		Then, the following algebraic identities hold for all \( 0 \leq s, u, t \leq t_0 \):
		\begin{align}\label{APOSOAa}
			\begin{split}
				\delta\tilde{X}_{s,t} &= \delta\tilde{X}_{s,u} + \delta\tilde{X}_{u,t}, \\
				\tilde{\mathbb{X}}_{s,t}^{2} &= \tilde{\mathbb{X}}_{s,u}^{2} + \tilde{\mathbb{X}}_{u,t}^{2} - \delta \tilde{X}_{u,t} \otimes \delta \tilde{X}_{s,u}, \\
				\tilde{\mathbb{X}}_{s,t}^{3} &= \tilde{\mathbb{X}}_{s,u}^{3} + \tilde{\mathbb{X}}_{u,t}^{3} - \tilde{\mathbb{X}}_{s,u}^{2} \otimes \delta \tilde{X}_{u,t} - \delta \tilde{X}_{u,t} \otimes \tilde{\mathbb{X}}_{s,u}^{2}.
			\end{split}
		\end{align}
		Assume that \( W \) is a finite-dimensional Banach space and let \( [a,b] \subseteq [0,T] \) with \( t_0 \geq b \). We say that \( (\tilde{Y}, \tilde{Y}^{(1)}, \tilde{Y}^{(2)}) \in \tilde{\mathcal{D}}^{\gamma}_{\tilde{\mathbf{X}}, W}([a,b]) \) if \( \tilde{Y} \), \( \tilde{Y}^{(1)} \), and \( \tilde{Y}^{(2)} \) are \( \gamma \)-Hölder continuous paths taking values in \( W \), \( \mathcal{L}(\mathbb{R}^d, W) \), and \( \mathcal{L}(\mathbb{R}^d \otimes \mathbb{R}^d, W) \), respectively, and satisfy:
		\begin{align}\label{expansion1}
			\begin{split}
				\delta \tilde{Y}_{s,t} - \tilde{Y}^{(1)}_t \delta\tilde{X}_{s,t} - \tilde{Y}^{(2)}_t \tilde{\mathbb{X}}^2_{s,t} 
				&= \tilde{Y}^{\#}_{s,t} = \mathcal{O}(|t-s|^{3\gamma}), \\
				\delta \tilde{Y}^{(1)}_{s,t} - \tilde{Y}^{(2)}_t \delta \tilde{X}_{s,t} 
				&= (\tilde{Y}^{(1)})^{\#}_{s,t} = \mathcal{O}(|t-s|^{2\gamma}),
			\end{split}
		\end{align}
		for all \( s, t \in [a,b] \), with \( s \leq t \).
		We claim that the following expression is well defined:
		\begin{align}\label{AOsasww}
			\int_{a}^{b}\tilde{Y}_{\tau}\mathrm{d}\tilde{\mathbf{X}}_\tau :=  \lim_{\substack{|\pi|\rightarrow 0,\\ \pi=\lbrace a = \tau_0 < \tau_{1} < \ldots < \tau_{k}=b \rbrace}}\sum_{0\leq j<k}\ \underbrace{\big[\tilde{Y}_{\tau_{j+1}} \delta\tilde{X}_{\tau_{j},\tau_{j+1}} + \tilde{Y}^{(1)}_{\tau_{j+1}} \tilde{\mathbb{X}}^2_{\tau_{j},\tau_{j+1}} + \tilde{Y}^{(2)}_{\tau_{j+1}} \tilde{\mathbb{X}}^3_{\tau_{j},\tau_{j+1}}\big]}_{\Xi_{\tau_j,\tau_{j+1}}}.
		\end{align}
		To see this, note that the identities in \eqref{APOSOAa}, combined with the expansion in \eqref{expansion1}, yield that for every \( a \leq \tau_1 \leq \tau_2 \leq \tau_3 \leq b \),
		\begin{align*}
			\Xi_{\tau_1,\tau_3} - \Xi_{\tau_1,\tau_2} - \Xi_{\tau_2,\tau_3}
			&= \tilde{Y}^{\#}_{\tau_2,\tau_3} \, \delta\tilde{X}_{\tau_1,\tau_2}
			+ (\tilde{Y}^{(1)})^{\#}_{\tau_2,\tau_3} \, \tilde{\mathbb{X}}^{2}_{\tau_1,\tau_2}
			+ \delta\tilde{Y}^{(2)}_{\tau_2,\tau_3} \, \tilde{\mathbb{X}}^{3}_{\tau_1,\tau_2} \\
			&= \mathcal{O}(|b-a|^{4\gamma}).
		\end{align*}
		Since \( 4\gamma > 1 \), the Sewing Lemma (cf.\ \cite{FH20}[Lemma 4.2]) implies that the definition in \eqref{AOsasww} is indeed well defined.
		It is now straightforward to see that, with slight reformulation, results analogous to Lemma~\ref{composit} and Lemma~\ref{expansion_2} can be established in the space \( \tilde{\mathcal{D}}^{\gamma}_{\tilde{\mathbf{X}}, W}([a,b]) \). Let \( (Y, Y^{(1)}, Y^{(2)}) \in \mathcal{D}^{\gamma}_{\mathbf{X}, W}([a,b]) \), and define
		\begin{align*}
			(\tilde{Y}, \tilde{Y}^{(1)}, \tilde{Y}^{(2)}) := (Y_{t_0 - \cdot}, Y^{(1)}_{t_0 - \cdot}, Y^{(2)}_{t_0 - \cdot}).
		\end{align*}
		Then, from \eqref{expansion1}, it follows immediately that
		\[
		(\tilde{Y}, \tilde{Y}^{(1)}, \tilde{Y}^{(2)}) \in \tilde{\mathcal{D}}^{\gamma}_{\tilde{\mathbf{X}}, W}([a,b]).
		\]
		Moreover, from \eqref{AOsasww}, we obtain
		\begin{align}\label{NMNMN}
			\int_{a}^{b} \tilde{Y}_{\tau} \mathrm{d}\tilde{\mathbf{X}}_\tau = -\int_{t_0 - b}^{t_0 - a} Y_{\tau} \mathrm{d}\mathbf{X}_\tau.
		\end{align}
		Based on these explanations, for all \( 0 \leq t \leq t_0 \), we set
		\[
		\big(\tilde{\phi}_{\tilde{\mathbf{X}}}^t(z_0),\tilde{\phi}_{\tilde{\mathbf{X}}}^t(z_0)^{(1)},\tilde{\phi}_{\tilde{\mathbf{X}}}^t(z_0)^{(2)}\big)=\big(\phi_{\mathbf{X}}^{t_0-t}(z_0),\phi_{\mathbf{X}}^{t_0-t}(z_0)^{(1)},\phi_{\mathbf{X}}^{t_0-t}(z_0)^{(2)}\big).
		\]
		Let \( (\tilde{\psi}_{\tilde{\mathbf{X}}}^{u})_{0 \leq u \leq t_0} \) be the solution to the equation
		\begin{align}\label{BNBBBB}
			&\mathrm{d}\tilde{\psi}_{\tilde{\mathbf{X}}}^{t}[\bar{z}]=D_{\tilde{\phi}_{\tilde{\mathbf{X}}}^t(z_0)}V \big(\tilde{\psi}_{\tilde{\mathbf{X}}}^{t}[\bar{z}]\big)\mathrm{d}\tilde{\mathbf{X}}_t-D_{\tilde{\phi}_{\tilde{\mathbf{X}}}^t(z_0)}V_0 \big(\tilde{\psi}_{\tilde{\mathbf{X}}}^{t}[\bar{z}]\big)\mathrm{d}t, \ \ \  \tilde{\psi}_{\tilde{\mathbf{X}}}^{0}[\bar{z}]=\Bar{z}\in \mathbb{R}^m.
		\end{align}
		We claim that \( \tilde{\psi}_{\tilde{\mathbf{X}}}^{t_0}[z] = (D_{z_0} \phi_{\mathbf{X}}^{t_0})^{-1}[z] \). To prove this, observe that from \eqref{Linearized}, for \( s, t \in [0, t_0] \) with \( s < t \) and for any \( z \in \mathbb{R}^m \),
		\begin{align*}
			\begin{split}
				&D_{z_0}\phi^{t_0-s}_{\mathbf{X}}[{z}]-D_{z_0}\phi^{t_0-t}_{\mathbf{X}}[z]\\ &\quad= \int_{t_0-t}^{t_0-s}D_{\phi^{\tau}_{\mathbf{X}}(z_0)}V\big( D_{z_0}\phi^{\tau}_{\mathbf{X}}[z]\big)\mathrm{d}\mathbf{X}_\tau +\int_{t_0-t}^{t_0-s} D_{\phi^{\tau}_{\mathbf{X}}(z_0)}V_0\big( D_{z_0}\phi^{\tau}_{\mathbf{X}}[z]\big) \, \mathrm{d}\tau .
			\end{split}
		\end{align*}
		From \eqref{NMNMN}
		\begin{align*}
			\begin{split}
				&D_{z_0}\phi^{t_0-t}_{\mathbf{X}}[{z}]-D_{z_0}\phi^{t_0-s}_{\mathbf{X}}[z]\\ &\quad= \int_{s}^{t}D_{\tilde{\phi}_{\tilde{\mathbf{X}}}^\tau(z_0)}V\big(D_{z_0}\phi^{t_0-\tau}_{\mathbf{X}}[z]\big)\mathrm{d}\tilde{\mathbf{X}}_\tau -\int_{s}^{t} D_{\tilde{\phi}_{\tilde{\mathbf{X}}}^\tau(z_0)}V_0\big(D_{z_0}\phi^{t_0-\tau}_{\mathbf{X}}[z]\big) \, \mathrm{d}\tau .
			\end{split}
		\end{align*}
		Therefore, \( D_{z_0} \phi^{t_0 - u}_{\mathbf{X}}[z] \) solves equation~\eqref{BNBBBB} on the interval \( [0, t_0] \) with initial condition \( \bar{z} = D_{z_0} \phi^{t_0}_{\mathbf{X}}[z] \). Thus, for all \( 0 \leq u \leq t_0 \),
		\begin{align*}
			\tilde{\psi}_{\tilde{\mathbf{X}}}^{u}\big[D_{z_0} \phi^{t_0}_{\mathbf{X}}[z]\big] = D_{z_0} \phi^{t_0-u}_{\mathbf{X}}[z].
		\end{align*}
		In particular, for \( u = t_0 \), we conclude that
		\[
		\tilde{\psi}_{\tilde{\mathbf{X}}}^{t_0}\big[D_{z_0} \phi^{t_0}_{\mathbf{X}}[z]\big] = z.
		\]
		Finally, we can proceed as in Proposition~\ref{DDSD}, using equation~\eqref{BNBBBB}, to show that
		\begin{align*}
			\sup_{0\leq t\leq T}\Vert{(D_{z_0}\phi_{\mathbf{X}}^t})^{-1}\Vert_{\mathcal{L}(\mathbb{R}^m,\mathbb{R}^m)}\leq \exp\big(Q_{1}(\Vert z_0\Vert,\Vert\mathbf{X}\Vert_{\gamma,[0,T]})\big).
		\end{align*}
	\end{proof}
We need an additional estimate.
\begin{proposition}\label{ZXZ}
	Let the conditions of Proposition~\ref{DDSD} hold. Also, suppose that for some \( 0 < r_1 \leq 1 \), there exists a polynomial \( P_2 \) such that
	\begin{align}\label{42}
		\forall z_1, z_0 \in \mathbb{R}^m: \quad \Vert D_{z_1} V_0 - D_{z_0} V_0 \Vert \leq P_2(\Vert z_1 \Vert, \Vert z_0 \Vert) \Vert z_1 - z_0 \Vert^{r_1}.
	\end{align}
	Fix \( \kappa \in \left( \frac{2}{3}, 1 \right) \). Then, there exists a polynomial \( Q_3 \) such that for \( \gamma \neq \frac{1}{3} \),
	\begin{align*}
		\forall z_1, z_0, \bar{z} \in \mathbb{R}^m \quad 
		&\vertiii{D_{z_1} \phi_{\mathbf{X}}[\bar{z}] - D_{z_0} \phi_{\mathbf{X}}[\bar{z}]}_{\mathcal{D}^{\gamma}_{\mathbf{X}}([0, T])} \\
		&\leq \max \left\{ \Vert z_1 - z_0 \Vert, \Vert z_1 - z_0 \Vert^{r_1}, \Vert z_1 - z_0 \Vert^{p-4}, \Vert z_1 - z_0 \Vert^{1 - \kappa} \right\} \\
		&\quad \times \exp\left( Q_3\left( \Vert z_1 \Vert, \Vert z_0 \Vert, \Vert \mathbf{X} \Vert_{\gamma, [0, T]} \right) \right) \Vert \bar{z} \Vert.
	\end{align*}
	\begin{proof}
		Recall that for \( \frac{1}{\gamma} + 1 < p < 5 \), the vector field \( V \) is of class \( \mathrm{Lip}^{p} \).  
		Define
		\[
		A(t, \mathbf{X}, z_1, z_0, \bar{z}) := D_{z_1} \phi^{t}_{\mathbf{X}}[\bar{z}] - D_{z_0} \phi^{t}_{\mathbf{X}}[\bar{z}].
		\]
		Then,
		\begin{align}\label{UJMNB}
			\begin{split}
				&A(t,\mathbf{X},z_1,z_0,\Bar{z})-A(s,\mathbf{X},z_1,z_0,\Bar{z})=\int_{s}^{t}D_{\phi^{\tau}_{\mathbf{X}}(z_1)}V\big( A(\tau,\mathbf{X},z_1,z_0,\Bar{z})\big)\mathrm{d}\mathbf{X}_\tau\\ &\quad + \int_{s}^{t}D_{\phi^{\tau}_{\mathbf{X}}(z_1)}V_0\big( A(\tau,\mathbf{X},z_1,z_0,\Bar{z})\big)\mathrm{d}\tau+
				\int_{s}^{t}\big(D_{\phi^{\tau}_{\mathbf{X}}(z_1)}V-D_{\phi^{\tau}_{\mathbf{X}}(z_0)}V \big)\big(  D_{z_0}\phi^{\tau}_{\mathbf{X}}[\Bar{z}]\big)\mathrm{d}\mathbf{X}_\tau\\ &\qquad+ \int_{s}^{t}\big(D_{\phi^{\tau}_{\mathbf{X}}(z_1)}V_0-D_{\phi^{\tau}_{\mathbf{X}}(z_0)}V_0 \big)  D_{z_0}\phi^{\tau}_{\mathbf{X}}[\Bar{z}]\mathrm{d}\tau.
			\end{split}
		\end{align}
		Therefore, from \eqref{bNNNMM},
		\begin{align}\label{NM49}
			\begin{split}
				&\vertiii{D_{\phi_{\mathbf{X}}(z_1)}V-D_{\phi_{\mathbf{X}}(z_0)}V}_{\mathcal{D}^{\kappa\gamma}_{\mathbf{X}}([0, T])}\\&\lesssim  \max\big\lbrace \vertiii{\phi_{\mathbf{X}}(z_1)-\phi_{\mathbf{X}}(z_0)}_{\mathcal{D}^{\gamma}_{\mathbf{X}}([0, T])},\vertiii{\phi_{\mathbf{X}}(z_1)-\phi_{\mathbf{X}}(z_0)}_{\mathcal{D}^{\gamma}_{\mathbf{X}}([0, T])}^{p-4},\vertiii{\phi_{\mathbf{X}}(z_1)-\phi_{\mathbf{X}}(z_0)}_{\mathcal{D}^{\gamma}_{\mathbf{X}}([0, T])}^{1-\kappa}\big\rbrace\\ &\times\big(1+\Vert\mathbf{X}\Vert_{\gamma,[0,T]}^9+ \vertiii{\phi_{\mathbf{X}}(z_1)}_{\mathcal{D}^{\gamma}_{\mathbf{X}}([0, T])}^9+\vertiii{\phi_{\mathbf{X}}(z_0)}_{\mathcal{D}^{\gamma}_{\mathbf{X}}([0, T])}^9\big) .
			\end{split}
		\end{align}
		Thus, by using \eqref{CDCa} and \eqref{NMUIO_1}, with \( \gamma_1 = \kappa \gamma \), we obtain
		\begin{align}\label{II0}
			\begin{split}
				[s,t]\subseteq[0,T]:\ \ &\vertiii{\int_{s}\big(D_{\phi^{\tau}_{\mathbf{X}}(z_1)}V-D_{\phi^{\tau}_{\mathbf{X}}(z_0)}V \big)\big(  D_{z_0}\phi^{\tau}_{\mathbf{X}}[\Bar{z}]\big)\mathrm{d}\mathbf{X}_\tau}_{\mathcal{D}^{\gamma}_{\mathbf{X}}([s,t])}\\
				&\lesssim (1+\Vert \mathbf{X}\Vert_{\gamma,[0,T]}^4)^2\vertiii{D_{\phi_{\mathbf{X}}(z_1)}V-D_{\phi_{\mathbf{X}}(z_0)}V}_{\mathcal{D}^{\kappa\gamma}_{\mathbf{X}}([0, T])}\vertiii{D_{z_0}\phi^{\tau}_{\mathbf{X}}[\Bar{z}]}_{\mathcal{D}^{\kappa\gamma}_{\mathbf{X}}([0, T])}\\	&\lesssim (1+\Vert \mathbf{X}\Vert_{\gamma,[0,T]}^4)^2\vertiii{D_{\phi_{\mathbf{X}}(z_1)}V-D_{\phi_{\mathbf{X}}(z_0)}V}_{\mathcal{D}^{\kappa\gamma}_{\mathbf{X}}([0, T])}\vertiii{D_{z_0}\phi^{\tau}_{\mathbf{X}}[\Bar{z}]}_{\mathcal{D}^{\gamma}_{\mathbf{X}}([0, T])}.
			\end{split}
		\end{align}
		Hence, by combining \eqref{NM49}, \eqref{II0}, Proposition~\ref{DDSD}, and Corollary~\ref{difference}, we deduce the existence of a polynomial \( Q_6 \) such that
		\begin{align}\label{II1}
			\begin{split}
				&[s,t]\subseteq [0,T]:\ \ \vertiii{\int_{s}\big(D_{\phi^{\tau}_{\mathbf{X}}(z_1)}V-D_{\phi^{\tau}_{\mathbf{X}}(z_0)}V \big)\big(  D_{z_0}\phi^{\tau}_{\mathbf{X}}[\Bar{z}]\big)\mathrm{d}\mathbf{X}_\tau}_{\mathcal{D}^{\gamma}_{\mathbf{X}}([s,t])}\\\quad &\leq\max\lbrace \Vert z_1-z_0\Vert,\Vert z_1-z_0\Vert^{p-4},\Vert z_1-z_0\Vert^{1-\kappa}\rbrace \exp\big(Q_6(\Vert z_1\Vert,\Vert z_0\Vert,\Vert\mathbf{X}\Vert_{\gamma,[0,T]})\big)\Vert\Bar{z}\Vert.
			\end{split}
		\end{align}
		Under the assumption imposed on \( V_0 \) in \eqref{42},
		\begin{align*}
			&[s,t]\subseteq [0,T]:\ \ \vertiii{\big(\int_{s}\big(D_{\phi^{\tau}_{\mathbf{X}}(z_1)}V_0-D_{\phi^{\tau}_{\mathbf{X}}(z_0)}V_0 \big)  D_{z_0}\phi^{\tau}_{\mathbf{X}}[\Bar{z}]\mathrm{d}\tau,0,0\big)}_{\mathcal{D}^{\gamma}_{\mathbf{X}}([s,t])}\\&\quad\lesssim (t-s)^{1-3\gamma}P_{2}\big(\vertiii{\phi_{\mathbf{X}}(z_1)}_{\mathcal{D}^{\gamma}_{\mathbf{X}}([0,T])},\vertiii{\phi_{\mathbf{X}}(z_0)}_{\mathcal{D}^{\gamma}_{\mathbf{X}}([0,T])}\big)\\&\quad \quad \times \vertiii{D_{z_0}\phi^{\tau}_{\mathbf{X}}[\Bar{z}]}_{\mathcal{D}^{\gamma}_{\mathbf{X}}([0, T])}\vertiii{\phi_{\mathbf{X}}(z_1)-\phi_{\mathbf{X}}(z_0)}_{\mathcal{D}^{\gamma}_{\mathbf{X}}([s,t])}^{r_1}.
		\end{align*}
		Consequently, from Proposition~\ref{DDSD} and Corollary~\ref{difference}, we can find a polynomial \( Q_7 \) such that
		\begin{align}\label{II2}
			\begin{split}
				&[s,t]\subseteq [0,T]:\ \ \vertiii{\big(\int_{s}\big(D_{\phi^{\tau}_{\mathbf{X}}(z_1)}V_0-D_{\phi^{\tau}_{\mathbf{X}}(z_0)}V_0 \big)  D_{z_0}\phi^{\tau}_{\mathbf{X}}[\Bar{z}]\mathrm{d}\tau,0,0\big)}_{\mathcal{D}^{\gamma}_{\mathbf{X}}([s,t])}\\&\quad\leq\Vert z_1-z_0\Vert^{r_1} \exp\big(Q_7(\Vert z_1\Vert,\Vert z_0\Vert,\Vert\mathbf{X}\Vert_{\gamma,[0,T]})\big) \Vert\Bar{z}\Vert.
			\end{split}
		\end{align}
		In summary, from \eqref{II1} and \eqref{II2}, we obtain			\begin{align}\label{II5}
			\begin{split}
				&\max\bigg\lbrace \vertiii{\int_{s}\big(D_{\phi^{\tau}_{\mathbf{X}}(z_1)}V-D_{\phi^{\tau}_{\mathbf{X}}(z_0)}V \big)\big(  D_{z_0}\phi^{\tau}_{\mathbf{X}}[\Bar{z}]\big)\mathrm{d}\mathbf{X}_\tau}_{\mathcal{D}^{\gamma}_{\mathbf{X}}([s,t])},\\&\myquad[3]\vertiii{\big(\int_{s}\big(D_{\phi^{\tau}_{\mathbf{X}}(z_1)}V_0-D_{\phi^{\tau}_{\mathbf{X}}(z_0)}V_0 \big)  D_{z_0}\phi^{\tau}_{\mathbf{X}}[\Bar{z}]\mathrm{d}\tau,0,0\big)}_{\mathcal{D}^{\gamma}_{\mathbf{X}}([s,t])}\bigg\rbrace
				\\ &\quad\leq \max\lbrace \Vert z_1-z_0\Vert,\Vert z_1-z_0\Vert^{r_1},\Vert z_1-z_0\Vert^{p-4},\Vert z_1-z_0\Vert^{1-\kappa}\rbrace \exp\big(Q_{8}(\Vert z_1\Vert,\Vert z_0\Vert,\Vert\mathbf{X}\Vert_{\gamma,[0,T]})\big)\Vert\Bar{z}\Vert,
			\end{split}
		\end{align}
		where \( Q_8 \) is a polynomial.
		From the assumption in~\eqref{41}, and similarly to~\eqref{DDrift},
		\begin{align}\label{II3}
			\begin{split}
				&[s,t]\subseteq [0,T]:\ \ \  \vertiii{\big(\int_{s}D_{\phi^{\tau}_{\mathbf{X}}(z_1)}V_0\big( A(\tau,\mathbf{X},z_1,z_0,\Bar{z})\big)\mathrm{d}\tau,0,0\big)}_{\mathcal{D}^{\gamma}_{\mathbf{X}}([s,t])}\\&\quad\lesssim 
				(t-s)^{1-3\gamma}P_{1}(\vertiii{\phi_{\mathbf{X}}(z_1)}_{\mathcal{D}^{\gamma}_{\mathbf{X}}([0,T])}) \vertiii{A(,\mathbf{X},z_1,z_0,\Bar{z})}_{\mathcal{D}^{\gamma}_{\mathbf{X}}([s,t])}.
			\end{split}
		\end{align}
		Analogous to \eqref{Sewing lemma}, we obtain the following estimate for a polynomial \( Q_9 \):
		\begin{align}\label{II4}
			\begin{split}
				&\vertiii{\int_{s}D_{\phi^{\tau}_{\mathbf{X}}(z_1)}V\big( A(\tau,\mathbf{X},z_1,z_0,\Bar{z})\big)\mathrm{d}\mathbf{X}_\tau}_{\mathcal{D}^{\gamma}_{\mathbf{X}}([s,t])}\lesssim (1+\Vert\mathbf{X}\Vert_{\gamma,[s,t]}^3)\Vert A(s,\mathbf{X},z_1,z_0,\Bar{z})\Vert\\&\quad+(t-s)^\gamma Q_{9}( \Vert\mathbf{X}\Vert_{\gamma,[0,T]},\vertiii{\phi_{\mathbf{X}}(z_1)}_{\mathcal{D}^{\gamma}_{\mathbf{X}}([0,T])}\Vert)\vertiii{A(,\mathbf{X},z_1,z_0,\Bar{z})}_{\mathcal{D}^{\gamma}_{\mathbf{X}}([s,t])}.
			\end{split}
		\end{align}
		We are now ready to derive our estimate. Assume that \( [s, t] \subseteq [0, T] \) with \( t - s \leq 1 \). From Proposition~\ref{RS17}, together with \eqref{II5} and \eqref{II4}, we conclude that there exists a constant \( M \geq 1 \) such that
		\begin{align}\label{II6}
			\begin{split}
				&\vertiii{A(,\mathbf{X},z_1,z_0,\Bar{z})}_{\mathcal{D}^{\gamma}_{\mathbf{X}}([s,t])}\leq M(1+\Vert\mathbf{X}\Vert_{\gamma,[0,T]}^3)\Vert A(s,\mathbf{X},z_1,z_0,\Bar{z})\Vert\\ & M\underbrace{\max\lbrace \Vert z_1-z_0\Vert,\Vert z_1-z_0\Vert^{r_1},\Vert z_1-z_0\Vert^{p-4},\Vert z_1-z_0\Vert^{1-\kappa}\rbrace \exp\big(Q_{8}(\Vert z_1\Vert,\Vert z_0\Vert,\Vert\mathbf{X}\Vert_{\gamma,[0,T]})\big)}_{B(\Vert z_1\Vert,\Vert z_0\Vert,\Vert\mathbf{X}\Vert_{\gamma,[0,T]})}\Vert\Bar{z}\Vert\\&+M(t-s)^{1-3\gamma}\bigg[Q_{9}( \Vert\mathbf{X}\Vert_{\gamma,[0,T]},\vertiii{\phi_{\mathbf{X}}(z_1)}_{\mathcal{D}^{\gamma}_{\mathbf{X}}([0,T])}\Vert)+P_{1}(\vertiii{\phi_{\mathbf{X}}(z_1)}_{\mathcal{D}^{\gamma}_{\mathbf{X}}([0,T])})\bigg]\times\\&\myquad[10]\vertiii{A(,\mathbf{X},z_1,z_0,\Bar{z})}_{\mathcal{D}^{\gamma}_{\mathbf{X}}([s,t])}.
			\end{split}
		\end{align}
		We now proceed as before. For \( 0 < \epsilon < 1 \), set
		\begin{align*}
			M\tau^{1-3\gamma} \bigg[Q_{9}( \Vert\mathbf{X}\Vert_{\gamma,[0,T]},\vertiii{\phi_{\mathbf{X}}(z_1)}_{\mathcal{D}^{\gamma}_{\mathbf{X}}([0,T])}\Vert)+P_{1}(\vertiii{\phi_{\mathbf{X}}(z_1)}_{\mathcal{D}^{\gamma}_{\mathbf{X}}([0,T])})\bigg]=1-\epsilon.
		\end{align*}
		We define a finite sequence $(\tau_{n})_{0\leq n\leq \bar{N}(\epsilon,\mathbf{X},z_1,z_0)}$ in \( [0, T] \), such that $\tau_0=0$ and $\tau_{\bar{N}(\epsilon,\mathbf{X},z_1,z_0)}=T$. For \( 0 \leq n < \bar{N}(\epsilon, \mathbf{X}, z_1, z_0) \), we impose \( \tau_{n+1} - \tau_n = \tau \). Therefore,
		\begin{align}\label{YUBN}
			\begin{split}
				&\bar{N}(\epsilon,\mathbf{X},z_1,z_0)=\lfloor\frac{T}{\tau}\rfloor+1\\ &\quad= \bigg\lfloor T\big(\frac{M\big(Q_{9}( \Vert\mathbf{X}\Vert_{\gamma,[0,T]},\vertiii{\phi_{\mathbf{X}}(z_1)}_{\mathcal{D}^{\gamma}_{\mathbf{X}}([0,T])}\Vert)+P_{1}(\vertiii{\phi_{\mathbf{X}}(z_1)}_{\mathcal{D}^{\gamma}_{\mathbf{X}}([0,T])})\big)}{1-\epsilon}\big)^{\frac{1}{1-3\gamma}}\bigg\rfloor +1.
			\end{split}
		\end{align}
		Set \( M_\epsilon = \frac{M}{\epsilon} \). Then, from \eqref{II6}, we have
		\begin{align*}
			0\leq n< \bar{N}(\epsilon,\mathbf{X},z_1,z_0):\ \ &\vertiii{A(,\mathbf{X},z_1,z_0,\Bar{z})}_{\mathcal{D}^{\gamma}_{\mathbf{X}}([\tau_n,\tau_{n+1}])}\leq M_\epsilon (1+\Vert\mathbf{X}\Vert_{\gamma,[0,T]}^3)\Vert A(\tau_n,\mathbf{X},z_1,z_0,\Bar{z})\Vert\\
			&\quad +M_\epsilon B(\Vert z_1\Vert,\Vert z_0\Vert,\Vert\mathbf{X}\Vert_{\gamma,[0,T]})\Vert\Bar{z}\Vert.
		\end{align*}
		In particular, since \( A(0, \mathbf{X}, z_1, z_0, \Bar{z}) = 0 \),
		\begin{align*}
			&\sup_{0\leq n< \bar{N}(\epsilon,\mathbf{X},z_1,z_0)} \vertiii{A(,\mathbf{X},z_1,z_0,\Bar{z})}_{\mathcal{D}^{\gamma}_{\mathbf{X}}([\tau_n,\tau_{n+1}])}\\&\quad \leq \sum_{0\leq n\leq \bar{N}(\epsilon,\mathbf{X},z_1,z_0)}\big(M_\epsilon (1+\Vert\mathbf{X}\Vert_{\gamma,[0,T]}^3)\big)^{n}M_\epsilon B(\Vert z_1\Vert,\Vert z_0\Vert,\Vert\mathbf{X}\Vert_{\gamma,[0,T]})\Vert\Bar{z}\Vert\\
			&\quad \lesssim \big(M_\epsilon (1+\Vert\mathbf{X}\Vert_{\gamma,[0,T]}^3)\big)^{\bar{N}(\epsilon,\mathbf{X},z_1,z_0)+1}B(\Vert z_1\Vert,\Vert z_0\Vert,\Vert\mathbf{X}\Vert_{\gamma,[0,T]})\Vert\Bar{z}\Vert.
		\end{align*}
		Recall from \eqref{DDFFD},
		\begin{align*}
			&\vertiii{A(,\mathbf{X},z_1,z_0,\Bar{z})}_{\mathcal{D}^{\gamma}_{\mathbf{X}}([\tau_n,T]}\leq\\ \quad& (\Vert\mathbf{X}\Vert_{\gamma,[0,T]}+1)^2\vertiii{A(,\mathbf{X},z_1,z_0,\Bar{z})}_{\mathcal{D}^{\gamma}_{\mathbf{X}}([\tau_n,\tau_{n+1}]}+\vertiii{A(,\mathbf{X},z_1,z_0,\Bar{z})}_{\mathcal{D}^{\gamma}_{\mathbf{X}}([\tau_{n+1},T]}.
		\end{align*}
		Consequently,
		\begin{align}\label{IN2}
			\begin{split}
				&\vertiii{A(,\mathbf{X},z_1,z_0,\Bar{z})}_{\mathcal{D}^{\gamma}_{\mathbf{X}}([0,T])}\lesssim (\bar{N}(\epsilon,\mathbf{X},z_1,z_0)+1)(\Vert\mathbf{X}\Vert_{\gamma,[0,T]}+1)^2\\ &\times\big(M_{\epsilon}(1+\Vert\mathbf{X}\Vert_{\gamma,[s,t]}^3)\big)^{\tilde{N}(\epsilon,\mathbf{X},z_0)+1} B(\Vert z_1\Vert,\Vert z_0\Vert,\Vert\mathbf{X}\Vert_{\gamma,[0,T]})\Vert\Bar{z}\Vert.
			\end{split}
		\end{align}
		From \eqref{BNNMM} and \eqref{YUBN}, it follows that $\bar{N}(\epsilon,\mathbf{X},z_1,z_0)$ exhibits polynomial growth with respect to $\Vert\mathbf{X}\Vert_{\gamma,[0,T]}$, $\Vert z_0\Vert$, and $\Vert z_1\Vert$. Consequently, applying inequality \eqref{IN2} establishes the desired result.
	\end{proof}
\end{proposition}
\begin{remark}
	In Proposition \ref{ZXZ}, we excluded the case $\gamma = \frac{1}{3}$ for technical reasons. However, if $\gamma = \frac{1}{3}$, one can simply work with any $\gamma'$ such that $\frac{1}{4} < \gamma' < \gamma$.
\end{remark}
	\begin{remark}   
		The implication of Proposition~\ref{ZXZ} is that the map $z_0 \mapsto D_{z_0} \phi_{\mathbf{X}}[\bar{z}]$ is H\"older continuous with respect to $z_0$. Moreover, the proposition yields a bound on the growth of this H\"older continuity, which, locally (in $z_0$), is less than $\exp(Q(\Vert \mathbf{X} \Vert_{\gamma, [0, T]}))$, where $Q$ is a polynomial function. As we will demonstrate in the next section, this result is a central tool in establishing the existence of invariant manifolds. In particular, the local bound it provides is sufficiently small for our analytical purposes. 
	\end{remark}
\section[Invariant manifolds and stability]{Invariant manifolds and stability} \label{manifolds}
In this section, we apply the estimates derived previously to establish the existence of the Lyapunov spectrum. Additionally, we demonstrate the existence of invariant manifolds. As a consequence, we derive pathwise exponential stability in a neighborhood of stationary points, assuming that all Lyapunov exponents are negative. We begin by recalling several fundamental definitions.
\begin{definition}
	Let $(\Omega,\mathcal{F},\mathbb{P})$ be a probability space, and let $\mathbb{T}$ denote either $\mathbb{Z}$ or $\mathbb{R}$. Suppose there exists a family of measurable maps $\{\theta_t\}_{t \in \mathbb{T}}$ on $\Omega$ satisfying the following properties:
	\begin{itemize}
		\item[(i)] $\theta_0 = \operatorname{id}$,
		\item[(ii)] for all $t, s \in \mathbb{T}$, we have $\theta_{t+s} = \theta_t \circ \theta_s$,
		\item[(iii)] if $\mathbb{T} = \mathbb{R}$, then the map $(t, \omega) \mapsto \theta_t \omega$ is $\mathcal{B}(\mathbb{R}) \otimes \mathcal{F} / \mathcal{F}$-measurable,
		\item[(iv)] for every $t \in \mathbb{T}$, we have $\mathbb{P} \circ \theta_t^{-1} = \mathbb{P}$.
	\end{itemize}
	Then the tuple $(\Omega, \mathcal{F}, \{\theta_t\}_{t \in \mathbb{T}}, \mathbb{P})$ is called an \emph{invertible measure-preserving dynamical system}. It is said to be \emph{ergodic} if, for every $t \in \mathbb{T} \setminus \{0\}$, the map $\theta_t : \Omega \to \Omega$ is ergodic.
\end{definition}
Another fundamental concept is the definition of a cocycle.
\begin{definition}\label{IOAOSLq}
	Let \( \mathcal{X} \) be a separable Banach space, and let \( (\Omega, \mathcal{F}, \{\theta_t\}_{t \in \mathbb{T}}, \mathbb{P}) \) be an invertible measure-preserving dynamical system. Let \( \mathbb{T}^{+} \subset \mathbb{T} \) denote the non-negative part of \( \mathbb{T} \). A map
	\begin{align*}
		\phi \colon \mathbb{T}^{+} \times \Omega \times \mathcal{X} \rightarrow \mathcal{X}
	\end{align*}
	is called a \emph{measurable cocycle} if it is jointly measurable and satisfies the cocycle property:
	\begin{align*}
		\forall s, t \in \mathbb{T}^{+},\ s < t: \ \phi(s+t, \omega, x) = \phi(s, \theta_{t} \omega, \phi(t, \omega, x)),
	\end{align*}
	We say that \( \phi \) is a \emph{\( C^k \)-cocycle} if, for every fixed \( (s, \omega) \in \mathbb{T}^{+} \times \Omega \), the map \( \phi(s, \omega, \cdot) \colon \mathcal{X} \rightarrow \mathcal{X} \) is of class \( C^k \). If \( \phi(s, \omega, \cdot) \) is linear for all \( (s, \omega) \), then \( \phi \) is called a \emph{linear cocycle}.
\end{definition}
In rough path theory, differential equations are solved pathwise. Consequently, this framework naturally aligns with the theory of random dynamical systems. We now recall some definitions and results from \cite{BRS17}, where the authors studied solutions of rough differential equations within the context of random dynamical systems.
\begin{definition}\label{BVCXZ}
	Let \( (\Omega, \mathcal{F}, \mathbb{P}, (\theta_t)_{t \in \mathbb{R}}) \) be an (ergodic) measure-preserving dynamical system. Let \( p \geq 1 \) and \( N \in \mathbb{N} \) satisfy \( p - 1 < N \leq p \). 
	
	A process 
	\[
	\mathbf{X} \colon \mathbb{R} \times \Omega \to T^{N}(\mathbb{R}^d) := \mathbb{R} \oplus \mathbb{R}^{d} \oplus \dots \oplus (\mathbb{R}^{d})^{\otimes N}
	\]
	is called a \( p \)-variation \emph{geometric rough path cocycle} if, for all \( \omega \in \Omega \) and every \( s, t \in \mathbb{R} \), the following conditions are satisfied:
	\begin{itemize}
		\item[(i)] \( \mathbf{X}(\omega) \) is a geometric \( p \)-variation rough path.
		\item[(ii)] We have
		\[
		\mathbf{X}_{s+t}(\omega) = \mathbf{X}_{s}(\omega) \otimes \mathbf{X}_{t}(\theta_s \omega).
		\]
		In particular, this implies  
		\[
		\mathbf{X}_{s, s+t}(\omega) = \mathbf{X}_{t}(\theta_s \omega).
		\]
	\end{itemize}
\end{definition}
We now return to the rough differential equation \eqref{SDE}. Throughout this section, we assume \( \frac{1}{4} < \gamma \leq \frac{1}{2} \), with particular focus on the case \( \frac{1}{4} < \gamma \leq \frac{1}{3} \), and that the equation is driven by a geometric \( \gamma \)-rough path cocycle \( \mathbf{X} \). It is known that solutions to such equations satisfy the flow property (see \cite{RS17}). Let \( \phi_{\mathbf{X}(\omega)}^t(x) \) denote the solution at time \( t \geq 0 \) with initial condition \( x \in \mathbb{R}^m \). Then, according to \cite[Theorem 21]{BRS17}, by defining
\[
\varphi_{\omega}^{t}(x) \coloneqq \phi_{\mathbf{X}(\omega)}^{t}(x),
\]
we obtain the cocycle property:
\[
\varphi^{t+s}_{\omega}(x) = \varphi^{t}_{\theta_s\omega} \left( \varphi_{\omega}^{s}(x) \right),
\]
for all \( s, t \geq 0 \) and \( x \in \mathbb{R}^m \).
The following definition generalizes the concept of a fixed point.
\begin{definition}\label{def:stationary_point}
	A random variable \( Y \colon \Omega \rightarrow \mathbb{R}^m \) is called a \emph{stationary point} if:
	\begin{itemize}
		\item[(i)] \( Y \) is measurable, and
		\item[(ii)] for every \( t > 0 \) and \( \omega \in \Omega \), we have \( \varphi^{t}_{\omega}(Y_{\omega}) = Y_{\theta_{t} \omega} \).
	\end{itemize}
\end{definition}

\begin{remark}
	Let \( Y \) be a stationary point. Then the linearized map, defined by
	\[
	\psi^{t}_{\omega} \colon \mathbb{R}^m \to \mathbb{R}^m, \quad
	z \mapsto \psi^{t}_{\omega}(z) \coloneqq D_{Y_{\omega}} \varphi^{t}_{\omega}[z],
	\]
	is a linear cocycle.
	
\end{remark}

For the remainder of this section, we impose the following additional assumption:
\begin{assumption}\label{S_U_C}
	\begin{itemize}
		\item [(i)] $(\Omega,\mathcal{F},\lbrace\theta_t\rbrace_{t\in\mathbb{T}},\P)$ is an invertible measure-preserving dynamical system.
		\item [(ii)] $(\Omega,\mathcal{F},\lbrace\theta_t\rbrace_{t\in\mathbb{T}},\P)$ is ergodic.
		\item[(iii)] For \( \frac{1}{4} < \gamma \leq \frac{1}{3} \), we assume that 
		\[
		\mathbf{X} = (X, \mathbb{X}^2, \mathbb{X}^3) \colon \mathbb{R} \times \Omega \to T^3(\mathbb{R}^d)
		\]
		is a geometric \( \gamma \)-rough path. Therefore, it is also a geometric rough path with finite \( \frac{1}{\gamma} \)-variation.
		\item [(iv)] We assume that for every $T > 0$,
		\begin{align*}
			\|\mathbf{X}(\omega)\|_{\gamma,[0,T]} \in \bigcap_{p\geq 1} \mathcal{L}^p(\Omega).
		\end{align*}
		\item [(iv)] We suppose that $\varphi$ admits a stationary point $Y$ such that
		\[
		\|Y(\omega)\|_{\gamma,[0,T]} \in \bigcap_{p\geq 1} \mathcal{L}^p(\Omega).
		\]
		\item [(v)] Assumption \ref{ASSVVNN} holds.
		\item [(vi)] We suppose that the assumptions imposed on $V_0$ and $V_1$ in Proposition \ref{DDSD} and Proposition \ref{ZXZ} are satisfied.
	\end{itemize}
\end{assumption}
\begin{remark}
	Fractional Brownian motions serve as prototypical examples of rough paths that satisfy the assumptions outlined above.
\end{remark}
\begin{remark}
	As stated earlier, all the results in this paper are also valid for the case where \( \frac{1}{3} < \gamma \leq \frac{1}{2} \). However, we focus on the more challenging regime \( \frac{1}{4} < \gamma \leq \frac{1}{3} \).
\end{remark}
\subsection{Invariant manifolds}
	Naturally, we expect that when we linearize the cocycle around the stationary point, the dynamical behavior of this linear cocycle allows us to locally analyze the behavior of the solutions. Before making these statements rigorous, let us first state the following result, which is a direct consequence of the multiplicative ergodic theorem.
\begin{proposition}\label{METT}
	For $t\geq 0$,  assume that \(\psi^{t}_{\omega} \coloneqq D_{Y_{\omega}}\varphi^{t}_{\omega} : \mathbb{R}^m \to \mathbb{R}^m\). Then, on a set of full measure \(\tilde{\Omega}\), invariant under \((\theta_t)_{t\in\mathbb{R}}\), there exists a sequence of deterministic values \(\mu_k < \ldots < \mu_1\), \(\mu_i \in [-\infty, \infty)\) (Lyapunov exponents), and \(m_i\)-dimensional subspaces \(H^i_\omega \subset \mathbb{R}^m\) such that:
	\begin{itemize}
		\item[(i)]	$\mathbb{R}^{m}=\bigoplus_{1\leq i\leq k} H^{i}_{\omega}$.
		\item[(ii)] For every \(1\leq i\leq k\), we have \(\psi^{t}_{\omega}( H^{i}_{\omega}) = H^{i}_{\theta_t\omega}\).
		\item[(iii)] \(\displaystyle\lim_{t\to \pm\infty} \frac{1}{t} \log\Vert \psi^{t}_{\omega}(\xi_{\omega})\Vert = \pm \mu_{i}\)
		if and only if \(\xi_{\omega} \in H^{i}_{\omega} \setminus \lbrace 0\rbrace\),
		where for \(t<0\), we set \(\psi^{t}_{\omega}(\xi_{\omega}) \coloneqq (\psi^{t}_{\omega})^{-1}(\xi_{\omega})\).
		
	\end{itemize}
\end{proposition}
\begin{proof}
	Fix \(t_0>0\), and let \(\log^{+}(x) = \max\lbrace\log x,0\rbrace\). From Proposition \ref{DDSD} and Lemma \ref{LLMMNNBB}, there exists a polynomial (depending on \(t_0\)) such that
	\begin{align*}
		&\max\bigg\lbrace\sup_{0\leq t\leq t_{0}}\log^{+}\big(\Vert \psi^{t}_{\omega}\Vert_{\mathcal{L}(\mathbb{R}^{m},\mathbb{R}^{m})}\big),\sup_{0\leq t\leq t_{0}}\log^{+}\big(\Vert (\psi^{t}_{\omega})^{-1}\Vert_{\mathcal{L}(\mathbb{R}^{m},\mathbb{R}^{m})}\big)\bigg\rbrace
		\\&\quad\leq Q_{1}(\Vert Y_\omega\Vert,\Vert \mathbf{X}(\omega)\Vert_{\gamma,[0,t_0]})\in\mathcal{L}^{1}(\Omega).
	\end{align*}
	Also, from Proposition \ref{DDSD},
	\begin{align}\label{BNMMK}
		\begin{split}
			&\sup_{0\leq t\leq t_0}\log^{+}\big(\Vert\psi^{t_0-t}_{\theta_t\omega}\Vert_{\mathcal{L}(\mathbb{R}^{m},\mathbb{R}^{m})}\big)=\sup_{0\leq t\leq t_0}\log^{+}\big(\Vert D_{Y_{\theta_t\omega}}\varphi^{t-t_0}\Vert_{\mathcal{L}(\mathbb{R}^{m},\mathbb{R}^{m})}\big)\\&\quad\leq\sup_{0\leq t\leq t_0} Q_{1}(\Vert Y_{\theta_t\omega}\Vert,\Vert \mathbf{X}(\omega)\Vert_{\gamma,[0,t_0]}).
		\end{split}
	\end{align}
	Since \(\varphi^{t}_{\omega}(Y_\omega) = Y_{\theta_t\omega}\), from \eqref{BNNMM},
	\begin{align}\label{475k}
		\sup_{0\leq t\leq t_ 0}\Vert Y_{\theta_t\omega}\Vert\leq Q(\Vert Y_\omega\Vert,\Vert\mathbf{X}\Vert_{\gamma,[0,t_0]}).
	\end{align}
	From \eqref{BNMMK} and \eqref{475k}, we conclude that
	\[
	\sup_{0\leq t\leq t_0} \log^{+}\!\bigl(\Vert \psi^{t_0-t}_{\theta_t\omega} \Vert_{\mathcal{L}(\mathbb{R}^{m},\mathbb{R}^{m})} \bigr)
	\]
	can be bounded by a polynomial function of \(\Vert Y_\omega\Vert\) and \(\Vert \mathbf{X} \Vert_{\gamma,[0,t_0]}\). Therefore
	\begin{align*}
		\sup_{0\leq t\leq t_0}\log^{+}\big(\Vert\psi^{t_0-t}_{\theta_t\omega}\Vert_{\mathcal{L}(\mathbb{R}^{m},\mathbb{R}^{m})}\big)\in\mathcal{L}^{1}(\Omega).
	\end{align*}
	The claims then follow from \cite[Theorem 3.4.11]{Arn98} or \cite[Theorem 1.21]{GVR23}.
\end{proof}
	This motivates the following definition:
	\begin{definition}
		We set (whenever these subspaces can be defined)		\begin{align}\label{linn}
			S_{\omega} \coloneqq \bigoplus_{i: \mu_{i}<0}H_{\omega}^{i}, \qquad  U_{\omega} \coloneqq \bigoplus_{i: \mu_{i}>0}H_{\omega}^{i}, \quad \text{and} \quad C_{\omega} \coloneqq H^{i_{c}}_{\omega} \quad \text{where} \quad \mu_{i_c}=0. 
		\end{align}
	\end{definition}
	\begin{remark}\label{STATEGY}
		As we will see in the sequel, there is a correspondence between the linear spaces defined in the preceding definition and the types of invariant manifolds. Specifically, the linear space \( S_{\omega} \) corresponds to the stable manifold, \( U_{\omega} \) to the unstable manifold, and \( C_{\omega} \) to the center manifold.  Although these manifolds are structurally distinct (as we will explain later), the primary technical conditions required for their existence are quite similar. Once these conditions are verified, we can invoke abstract results to directly establish the existence of these manifolds. Proposition~\ref{ZXZ} provides the essential tools for verifying these conditions.
		Let us outline the main strategy that is sufficient to establish the existence of these manifolds. We fix \( t_0 > 0 \) and define
		\[
		P_{\omega}(z) := \varphi^{t_0}_{\omega}(z + Y_\omega) 
		- \varphi^{t_0}_{\omega}(Y_\omega) 
		- \psi^{t_0}_{\omega}(z).
		\]
		We then verify that
		\begin{align} \label{eqn:diff_bound_P}
			\Vert P_{\omega}(z_2) - P_{\omega}(z_1) \Vert 
			\leq \Vert z_2 - z_1 \Vert \cdot f(\omega) \cdot h\!\bigl(\Vert z_2 \Vert + \Vert z_1 \Vert\bigr),
		\end{align}
		almost surely, where \( f: \Omega \rightarrow \mathbb{R}^{+} \) is a measurable function satisfying
		\[
		\lim_{n \rightarrow \infty} \frac{1}{n} \log^{+} f(\theta_{n t_0} \omega) = 0 
		\quad \text{almost surely},
		\]
		and \( h(x) = x^{r} g(x) \) for some \( r > 0 \), with \( g: \mathbb{R} \rightarrow \mathbb{R}^{+} \) an increasing \( C^{1} \) function. 
		Once these conditions are verified, the abstract results of \cite{GVR23} (for the stable and unstable manifolds) and \cite{GVR25B} (for the center manifold) are robust enough to guarantee the existence of such manifolds.
	\end{remark}
We now state the first result concerning the existence of stable manifolds.
\begin{theorem}[Local stable manifolds]\label{stable_manifold}
Suppose that in Proposition~\ref{METT}, we have \( \mu_i < 0 \) for some \( k \leq i \leq 1 \). Fix an arbitrary time step \( t_{0} > 0 \) and let \( 0 < \nu < -\mu^{-} \), where
\[
\mu^{-} := \max \bigl\lbrace \mu_{i} : \mu_{i} < 0 \bigr\rbrace.
\]
Then there exists a $\theta_{t_0}$-invariant subset $\tilde{\Omega} \subseteq \Omega$ of full measure, and a family of immersed submanifolds $S^{\nu}_{\text{loc}}(\omega) \subset \mathbb{R}^{m}$ such that for every $\omega \in \tilde{\Omega}$, the following properties hold:
\begin{itemize}
	\item [(i)] There exist two positive and finite random variables $\rho^{\nu}_{1}(\omega), \rho^{\nu}_{2}(\omega)$ such that
	\begin{align*}
		\liminf_{p \to \infty} \frac{1}{p} \log \rho_i^{\nu}(\theta_{pt_0}\omega)\geq 0,\  i = 1,2
	\end{align*}
	and 
	\begin{align}\label{UIOPL}
		\begin{split}
			\big{\lbrace} z \in \mathbb{R}^{m} \, :\, \sup_{n\geqslant 0}\exp(nt_0\nu)&\Vert \varphi^{nt_0}_{\omega}(z) - Y_{\theta_{nt_0}\omega}\Vert <\rho_{1}^{\nu}(\omega)\big{\rbrace}\subseteq S^{\nu}_{loc}(\omega)\\&\subseteq \big{\lbrace} z \in \mathbb{R}^{m} \, : \sup_{n\geqslant 0}\exp(nt_0\nu)\Vert \varphi^{nt_0}_{\omega}(z) - Y_{\theta_{nt_0}\omega}\Vert<\rho_{2}^{\nu}({\omega})\big{\rbrace}.
		\end{split}
	\end{align}
	\item [(ii)]  $T_{Y_\omega}S^{\nu}_{loc}(\omega)=S_{\omega}$ and for $n\geq N(\omega):\varphi^{nt_0}_{\omega}(S^{\nu}_{loc}(\omega))\subset S^{\nu}_{loc}(\theta_{nt_0}\omega)$. 
	\item [(iii)] For $0 < \nu_{1} \leq \nu_{2} < -\mu^{-}$, we have $S^{\nu_{2}}_{loc}(\omega) \subseteq S^{\nu_{1}}_{loc}(\omega)$, and for $n \geq N(\omega)$,
	\[
	\varphi^{nt_0}_{\omega}(S^{\nu_{1}}_{loc}(\omega)) \subseteq S^{\nu_{2}}_{loc}(\theta_{nt_0}\omega).
	\]
	Therefore, for every $z \in S^{\nu}_{loc}(\omega)$,
	\[
	\limsup_{n \to \infty} \frac{1}{n} \log \Vert \varphi^{nt_0}_{\omega}(z) - Y_{\theta_{nt_0}\omega} \Vert \leq t_0 \mu^{-}.
	\]
	Moreover,
	\begin{align*}
		\limsup_{n\rightarrow\infty} \frac{1}{n} \log\bigg{[}\sup\bigg{\lbrace}\frac{\Vert\varphi^{nt_0}_{\omega}(\tilde{z}) - \varphi^{nt_0}_{\omega}(z)  \Vert }{\Vert \tilde{z}-z\Vert},\ \ \tilde{z} \neq {z},\ \ \text{and}\ \  \tilde{z},{z}\in S^{\nu}_{loc}(\omega) \bigg{\rbrace}\bigg{]}\leq t_0\mu^{-}.
	\end{align*}
\end{itemize}
\end{theorem}
\begin{proof}
Following the explanation in Remark~\ref{STATEGY}, our aim is to apply \cite[Theorem 2.10]{GVR23}. Define
\begin{align*}
	P_{\omega}({z}):=\varphi^{t_0}_{\omega}(z+Y_\omega)-\varphi^{t_0}_{\omega}(Y_\omega)-\psi^{t_0}_{\omega}(z).
\end{align*}
For the same $\kappa$ as in Proposition \ref{ZXZ}, set
\begin{align*}
	T(z_1,z_0):=\max\lbrace \Vert z_1\Vert+\Vert z_0\Vert,\Vert z_1\Vert^{r_1}+\Vert z_0\Vert^{r_1},\Vert z_1\Vert^{p-4}+\Vert z_0\Vert^{p-4},\Vert z_1\Vert^{1-\kappa}+\Vert z_0\Vert^{1-\kappa}\rbrace.
\end{align*}
From Proposition \ref{ZXZ}, there exist a polynomial $\tilde{Q}$ and an increasing $C^1$ function $g : \mathbb{R} \to (0, \infty)$ such that
\begin{align*}
	&\big\Vert P_{\omega}(z_1)-P_{\omega}(z_0)\big\Vert\leq\int_{0}^{1}\big\Vert (D_{\theta z_1+(1-\theta)z_0+Y_\omega}\varphi^{t_0}_{\omega}-D_{Y_\omega}\varphi^{t_0}_{\omega}\big)[z_1-z_0]\big\Vert\mathrm{d}\theta\\&\quad\leq \exp\big(\tilde{Q}(\Vert Y_\omega\Vert,\Vert\mathbf{X}(\omega)\Vert_{\gamma,[0,t_0]})\big) g(\Vert z_1\Vert+\Vert z_0\Vert)T(z_1,z_0)\Vert z_1-z_0\Vert.
\end{align*}
From our assumption,
\begin{align*}
	f(\omega)=\log\big(\exp\big(\tilde{Q}(\Vert Y_\omega\Vert,\Vert\mathbf{X}(\omega)\Vert_{\gamma,[0,t_0]})\big)\big)=\tilde{Q}(\Vert Y_\omega\Vert,\Vert\mathbf{X}(\omega)\Vert_{\gamma,[0,t_0]})\in\mathcal{L}^{1}(\Omega).
\end{align*}
Hence, by Birkhoff's ergodic theorem, on a set of full measure,
\begin{align*}
	\lim_{n\rightarrow\infty} \frac{1}{n} \log^{+}f(\theta_{nt_0}\omega)=0.
\end{align*}
Therefore, we can apply \cite[Theorem 2.10]{GVR23} to obtain the result.	
\end{proof}
\begin{remark}\label{cocnnn}
A natural question is whether a continuous-time version of Theorem \ref{stable_manifold} can be deduced. It turns out that a slightly weaker result can be obtained for the continuous-time case. We briefly outline the procedure. Let $0 \leq t_1 < t_0$ and $z \in S^{\nu}_{loc}(\omega)$. By the cocycle property,
\begin{align*}
	&\varphi^{nt_0+t_1}_{\omega}(z)-\varphi^{nt_0+t_1}_{\omega}(Y_{\omega})=\varphi^{t_1}_{\theta_{nt_0}\omega}(\varphi^{nt_0}_{\omega}(z))-\varphi^{t_1}_{\theta_{nt_0}\omega}(\varphi^{nt_0}_{\omega}(Y_{\omega})).
\end{align*} 
Consequently, from Corollary \ref{difference},	\begin{align}\label{continous_part}
	\begin{split}
		&\sup_{0\leq t_1< t_0}\Vert \varphi^{nt_0+t_1}_{\omega}(z)-\varphi^{nt_0+t_1}_{\omega}(Y_{\omega})\Vert\\ \quad&\leq \Vert \varphi^{nt_0}_{\omega}(z)-Y_{\theta_{nt_0}\omega}\Vert \exp\big(Q_2(\Vert \varphi^{nt_0}_{\omega}(z)\Vert,\Vert Y_{\theta_{nt_0}\omega}\Vert,\Vert\mathbf{X}(\theta_{nt_0}\omega)\Vert_{\gamma,[0,t_0]})\big).
	\end{split}
\end{align}
Recall that $z \in S^{\nu}_{loc}(\omega)$ and
\begin{align*}
	\Vert \varphi^{nt_0}_{\omega}(z)\Vert\leq \Vert \varphi^{nt_0}_{\omega}(z)-Y_{\theta_{nt_0}\omega}\Vert+\Vert Y_{\theta_{nt_0}\omega}\Vert.
\end{align*}
Thus, by Birkhoff's ergodic theorem, on a set of full measure,	\begin{align}\label{NMMMKKIIII}
	\lim_{n\rightarrow\infty}\frac{1}{n}Q_2(\Vert \varphi^{nt_0}_{\omega}(z)\Vert,\Vert Y_{\theta_{nt_0}\omega}\Vert,\Vert\mathbf{X}(\theta_{nt_0}\omega)\Vert_{\gamma,[0,t_0]})= 0 .
\end{align}
Let $t > 0$, with $t = m t_0 + t_1$, where $0 \leq t_1 < t_0$ and $0 < \nu_1 < \nu$. Then,
\begin{align}\label{CXZAQWE}
	\begin{split}
		&\sup_{n\geq 0}\exp(nt_0\nu_1)\Vert\varphi^{nt_0}_{\theta_{t}\omega}(\varphi^{t}_{\omega}(z))-Y_{\theta_{nt_0+t}\omega}\Vert\leq \big(\sup_{k\geq 0}\exp(kt_0\nu)\Vert\varphi^{kt_0}_{\omega}(z)-Y_{\theta_{kt_0}\omega}\Vert\big)\\
		&\times \underbrace{\sup_{n\geq 0}\bigg(\exp(-(m+n)t_{0}\nu+nt_{0}\nu_{1})\exp\big(Q_2(\Vert \varphi^{(m+n)t_0}_{\omega}(z)\Vert,\Vert Y_{\theta_{(m+n)t_0}\omega}\Vert,\Vert\mathbf{X}(\theta_{(m+n)t_0}\omega)\Vert_{\gamma,[0,t_0]})\big)\bigg)}_{I(\omega,t,z)}.
	\end{split}
\end{align}
%
Since $0 < \nu_1 < \nu$, it follows from \eqref{NMMMKKIIII} that for sufficiently large $t > t(\omega)$, the term $I(\omega, t, z)$ can be made arbitrarily small. Hence, from \eqref{UIOPL},

\begin{align*}
	\varphi^{t}_{\omega}(S^{\nu_{1}}_{loc}(\omega))\subseteq S^{\nu}_{loc}(\theta_{t}\omega).
\end{align*}
\end{remark}
\begin{remark}
	In the previous remark, we justified that when we discretize the cocycle and use the existence of invariant stable manifolds for the discretized cocycle, we can indeed derive a continuous-time version of the invariant stable manifolds (albeit in a weaker sense). This is possible primarily due to Corollary \ref{difference}, which plays a crucial role in the argument.
	Specifically, during the estimation in \eqref{continous_part}, it ensures that the solution in a neighbourhood of the stationary point does not grow excessively in the time intervals between successive discrete times. More precisely, the term
	\[
	\exp\Big(
	Q_2\big(
	\Vert \varphi^{nt_0}_{\omega}(z) \Vert,
	\Vert Y_{\theta_{nt_0}\omega} \Vert,
	\Vert \mathbf{X}(\theta_{nt_0}\omega) \Vert_{\gamma, [0, t_0]}
	\big)
	\Big)
	\]
	turns out to be a tempered random variable. Therefore, since we are lying within the stable manifolds at the discrete times, the solution approaches the stationary point exponentially. Consequently, this additional bound does not pose a significant problem and can be effectively controlled over large time horizons.
\end{remark}
The next result concerns the existence of unstable manifolds.
\begin{theorem}[Local unstable manifolds]\label{Local unstable manifolds}
Suppose that in Proposition~\ref{METT}, we have \( \mu_{1} > 0 \). Let \( 0 < \nu < \mu^{+} \), where
\[
\mu^{+} := \min \bigl\lbrace \mu_{i} : \mu_{i} > 0 \bigr\rbrace,
\]
and fix an arbitrary time step \( t_{0} > 0 \). Then there exists a $\theta_{t_0}$-invariant subset $\tilde{\Omega} \subseteq \Omega$ of full measure, and a family of immersed submanifolds $U^{\nu}_{loc}(\omega)$ of $\mathbb{R}^{m}$ such that, for every $\omega \in \tilde{\Omega}$:
\begin{itemize}
	\item [(i)] There exist two positive and finite random variables $\tilde{\rho}^{\nu}_{1}(\omega), \tilde{\rho}^{\nu}_{2}(\omega)$ such that
	\[
	\liminf_{p \to -\infty} \frac{1}{p} \log \tilde{\rho}_i^{\nu}(\theta_{pt_0} \omega) \geq 0, \quad i = 1,2,
	\]
	and
	\begin{align*}
		&\bigg{\lbrace} z_\omega \in \mathbb{R}^{m} \, :\,\exists \lbrace z_{\theta_{-nt_{0}}\omega}\rbrace_{n\geq 1} \text{ s.t. } \varphi^{mt_{0}}_{\theta_{-nt_{0}}\omega}(z_{\theta_{-nt_{0}}\omega}) = z_{\theta_{(m-n)t_{0}}\omega} \text{ for all } 0 \leq m \leq n \text{ and }\\ &\sup_{n\geq 0}\exp(nt_0\nu)\Vert z_{\theta_{-nt_{0}}\omega} - Y_{\theta_{-nt_{0}}\omega} \Vert <\tilde{\rho}_{1}^{\nu}(\omega)\bigg{\rbrace}\subseteq U^{\nu}_{loc}(\omega)\subseteq \bigg{\lbrace} z_\omega \in \mathbb{R}^{m} \, :\,\exists \lbrace z_{\theta_{-nt_{0}}\omega}\rbrace_{n\geq 1} \text{ s.t. }\\& \varphi^{mt_{0}}_{\theta_{-nt_{0}}\omega}(z_{\theta_{-nt_{0}}\omega}) = z_{\theta_{(m-n)t_{0}}\omega} \text{ for all } 0 \leq m \leq n \text{ and }\sup_{n\geq 0}\exp(nt_0\nu)\Vert z_{\theta_{-nt_{0}}\omega} - Y_{\theta_{-nt_{0}}\omega} \Vert<\tilde{\rho}_{2}^{\nu}(\omega)\bigg{\rbrace}
	\end{align*}
	\item [(ii)]  $T_{Y_\omega}U^{\nu}_{loc}(\omega)=U_{\omega}$ and for $n\geq N(\omega): U^{\nu}_{loc}(\omega)\subset \varphi^{nt_0}_{\theta_{-nt_0}\omega}(U^{\nu}_{loc}(\theta_{-nt_0}\omega))$. 
	\item [(iii)] For $0 < \nu_{1} \leq \nu_{2} < \mu^{+}$, we have $U^{\nu_{2}}_{loc}(\omega) \subseteq U^{\nu_{1}}_{loc}(\omega)$, and for $n \geq N(\omega)$,
	\[
	U^{\nu_{1}}_{loc}(\omega) \subseteq \varphi^{nt_0}_{\theta_{-nt_0}\omega}(U^{\nu_{2}}_{loc}(\theta_{-nt_0}\omega)).
	\]
	Therefore, for every $z_\omega \in U^{\nu}_{loc}(\omega)$,
	\[
	\limsup_{n \rightarrow \infty} \frac{1}{n} \log \Vert z_{\theta_{-nt_0}\omega} - Y_{\theta_{-nt_0}} \Vert \leq -t_0 \mu^{+}.
	\]
	In addition,
	\begin{align*}
		\limsup_{n\rightarrow -\infty} \frac{1}{n} \log\bigg{[}\sup\bigg{\lbrace}\frac{\Vert\tilde{z}_{\theta_{-nt_0}\omega} - {z}_{\theta_{-nt_0}\omega}  \Vert }{\Vert \tilde{z}_\omega-z_\omega\Vert},\ \ \tilde{z}_\omega \neq z_\omega,\ \ \text{and}\ \ \tilde{z}_\omega,z_\omega\in U^{\nu}_{loc}(\omega) \bigg{\rbrace}\bigg{]}\leq -t_0\mu^{+}.
	\end{align*}
\end{itemize}
\end{theorem}
\begin{proof}
	The proof is almost identical to that of Theorem~\ref{stable_manifold}. 
	We apply the same strategy (see Remark~\ref{STATEGY}), and then invoke \cite[Theorem~2.17]{GVR23} to obtain the result.
\end{proof}
\begin{remark}
Similar to Remark~\ref{cocnnn}, we can also obtain a continuous-time result for the unstable manifold. For $z_\omega \in U^{\nu}_{\mathrm{loc}}(\omega)$, let $\lbrace z_{\theta_{-nt_0}\omega} \rbrace_{n \geq 1}$ be the corresponding sequence from item~(i) of Theorem~\ref{Local unstable manifolds}. Define
\begin{align*}
	\text{for} \ \  nt_0\leq t<(n+1)t_0: \ \ \ z_{\theta_{-t}\omega}:=\varphi^{(n+1)t_0-t}_{\theta_{-(n+1)t_0}\omega}(z_{\theta_{-(n+1)t_0}\omega}). 
\end{align*}
From Corollary \ref{difference},
\begin{align*}
	&\Vert z_{\theta_{-t}\omega} - Y_{\theta_{-t}\omega} \Vert = \Vert\varphi^{(n+1)t_0-t}_{{-(n+1)t_0}\omega}(z_{\theta_{-(n+1)t_0}\omega})-\varphi^{(n+1)t_0-t}_{\theta_{-(n+1)t_0}\omega}(Y_{\theta_{-(n+1)t_0}\omega})\Vert\\
	&\qquad \leq\Vert z_{\theta_{-(n+1)t_0}\omega}-Y_{\theta_{-(n+1)t_0}\omega}\Vert \exp\big(Q_2(\Vert z_{\theta_{-(n+1)t_0}\omega}\Vert,\Vert Y_{\theta_{-(n+1)t_0}\omega}\Vert,\Vert\mathbf{X}(\theta_{-(n+1)t_0}\omega)\Vert_{\gamma,[0,t_0]})\big).
\end{align*}
Thanks to Birkhoff's ergodic theorem, on a set of full measure, one concludes that
\begin{align*}
	\lim_{n \rightarrow \infty} \frac{1}{n} Q_2\left( \| z_{\theta_{-(n+1)t_0} \omega} \|, \| Y_{\theta_{-(n+1)t_0} \omega} \|, \| \mathbf{X}(\theta_{-(n+1)t_0} \omega) \|_{\gamma,[0,t_0]} \right) = 0.
\end{align*}
Consequently, by an argument similar to \eqref{CXZAQWE} and the lines below, if $\nu_1 < \nu$ and $t \geq t(\omega)$, we have
\begin{align*}
	U^{\nu_1}_{loc}(\omega)\subseteq \varphi^{t}_{\theta_{-t}\omega}(U^{\nu}_{loc}(\theta_{-t}\omega).
\end{align*}
\end{remark}
\begin{remark}
	Theorem \ref{Local unstable manifolds} should be interpreted in the following way: on the unstable manifold, at the current state, one can move backward in negative time and approach the stationary point arbitrarily closely. 
	Due to the way the unstable manifold is formulated, we can drop the assumption that the cocycle must be injective. This is particularly important in the context of cocycles defined on Banach spaces, where injectivity is not necessarily guaranteed (cf.~\cite{GVR21, GVR24, GVR25A}).
\end{remark}
Our final result concerns the existence of center manifolds.
\begin{theorem}[Local center manifolds]\label{center_manifold} 
	In Proposition~\ref{METT}, suppose that for some index \(1 \leq i_c \leq k\), we have \(\mu_{i_c} = 0\). 
	Fix an arbitrary time step \(t_0 > 0\) and assume
	\[
	0 < \nu < \min\!\bigl\lbrace \mu_{i_c - 1}, \, -\mu_{i_c + 1} \bigr\rbrace,
	\]
	with the convention that \(\mu_0 = \infty\) if \(i_c = 1\).
	Let $\mathbb{N}_0 t_0 = \bigl\{ n t_0 : n \in \mathbb{N}_0 = \mathbb{N} \cup \lbrace 0 \rbrace \bigr\}.$
	Then there exists a \(\theta_{t_0}\)-invariant subset \(\tilde{\Omega} \subseteq \Omega\) of full measure, a continuous cocycle
	\[
	\bar{\varphi} \colon \mathbb{N}_0 t_0 \times \tilde{\Omega} \times \mathbb{R}^m \to \mathbb{R}^m,
	\]
	and a positive random variable \(\rho^c \colon \tilde{\Omega} \to (0, \infty)\) such that
	\[
	\liminf_{n \to \pm\infty} \frac{1}{n} \log \rho^c\!\bigl(\theta_{n t_0} \omega\bigr) \geq 0.
	\]
	Moreover, for the functions \(\rho^c\) and \(\bar{\varphi}\), we have
	\[
	\| z - Y_\omega \| \leq \rho^c(\omega)
	\implies
	\bar{\varphi}^{t_0}_\omega(Y_\omega + z)
	= \varphi^{t_0}_\omega(Y_\omega + z).
	\]
	In addition, for every \(\omega \in \tilde{\Omega}\), there exists a function
	\[
	h^c_\omega \colon C_\omega \to \mathcal{M}^{c,\nu}_\omega \subset \mathbb{R}^m,
	\]
	such that:
	\begin{itemize}
		\item[(i)] $h^c_{\omega}$ is a homeomorphism and is Lipschitz continuous.
		\item[(ii)] For any $k \in \mathbb{N}$, if the map $\varphi^{t_0}_{\omega}$ is of class $C^k$, then the manifold $\mathcal{M}^{c,\nu}_{\omega}$ is of class $C^{k-1}$.
		\item[(iii)] $\mathcal{M}^{c,\nu}_{\omega}$ is $\bar{\varphi}$-invariant; i.e., for every $n \in \mathbb{N}_0$, we have
		\[
		\bar{\varphi}^{nt_0}_{\omega}(\mathcal{M}^{c,\nu}_{\omega}) \subseteq \mathcal{M}^{c,\nu}_{\theta_{nt_0}\omega}.
		\]
		\item[(iv)] For every $z_\omega \in \mathcal{M}^{c,\nu}_{\omega}$, there exists a sequence $\lbrace z_{\theta_{-nt_0}\omega} \rbrace_{n \geq 1}$ such that, if we define
		\begin{align*}
			j\leq 0:\ \  \bar{\varphi}^{jt_0}_{\omega}(z_\omega):=z_{\theta_{jt_0}\omega},
		\end{align*}
		then 
		\begin{align*}
			\forall(k,j)\in\mathbb{N}_{0}\times\mathbb{Z}:\ \ \bar{\varphi}^{kt_0}_{\theta_{jt_0}\omega}(\bar{\varphi}^{jt_0}_{\omega}(z_\omega))=\bar{\varphi}^{(k+j)t_0}_{\omega}(z_\omega).
		\end{align*}
		Also,
		\begin{align*}
			\sup_{j\in\mathbb{Z}}\exp(-\nu\vert j\vert)\Vert\bar{\varphi}^{jt_0}_{\omega}(z_\omega)-Y_{\theta_{jt_0}\omega}\Vert<\infty .
		\end{align*}
	\end{itemize}
\end{theorem}
\begin{proof}
	Thanks to Remark~\ref{STATEGY}, we argue as in the proof of Theorem~\ref{stable_manifold}, and then apply \cite[Theorem~3.14]{GVR25B}.
\end{proof}
\begin{remark}
Assume that for \( \omega \in \tilde{\Omega} \), the function \( \varphi_{\omega} \) is \( C^{k} \). Then, by \cite[Remark~2.11, Remark~2.18]{GVR23} and \cite[Theorem~3.14]{GVR25B}, the corresponding invariant manifolds (stable, unstable, and center) are of class \( C^{k-1} \).
\end{remark}
\subsection{Exponential stability}\label{sec:stability}
The stable manifold chart is a local homeomorphism between \( S_\omega \) and \( S^{\nu}_{\text{loc}}(\omega) \). Therefore, when all the Lyapunov exponents are negative, it is natural to expect that, in a neighborhood of the stationary point, the solutions decay exponentially toward the stationary point. We refer to this result as \emph{local stability} and formalize it in the following corollary.
\begin{corollary}[Local stability]\label{stability}
Consider the situation of Theorem~\ref{stable_manifold}, with the additional assumption that \(\mu_1 < 0\). Then there exists a positive random variable \(R^{\nu}(\omega) > 0\) such that
\[
\liminf_{p\rightarrow\infty} \frac{1}{p}\log R^{\nu}(\theta_{pt_0}) \geq 0,
\]
and
\[
\bigl\lbrace z\in\mathbb{R}^m : \Vert z - Y_\omega\Vert < R^{\nu}(\omega) \bigr\rbrace 
= S^{\nu}_{\text{loc}}(\omega).
\]
In addition, on a set of full measure \(\tilde{\tilde{\Omega}} \subseteq \tilde{\Omega}\), 
for every \(0 < \nu_1 < \nu\) and every \(z \in \mathbb{R}^m\) satisfying 
\(\| z - Y_\omega \| < R^{\nu}(\omega)\), we have
\begin{align*}
	\sup_{t\geq 0}\exp(t\nu_1)\Vert\varphi^{t}_{\omega}(z)-Y_{\theta_{t}\omega} \Vert<\infty.
\end{align*}
\end{corollary}
\begin{proof}
First, we need to slightly modify two results from \cite{GVR23}. 
In that paper, the results are stated in a very general form; however, since we do not require such generality, we will adapt the notation to fit the context of our current manuscript. From our assumption, we have
\[
F_{\mu_1}(\omega) = \bigl\lbrace z \in \mathbb{R}^m : 
\limsup_{n \to \infty} \frac{1}{n t_0} \log \| \psi^{n t_0}_\omega(z) \| 
\leq \mu_1 \bigr\rbrace = S_\omega = \mathbb{R}^m.
\]
From \cite[Theorem 2.10]{GVR23},
\[
S^\nu_{\mathrm{loc}}(\omega) 
= \bigl\lbrace Y_\omega + \Pi^0\!\bigl(\Gamma(z)\bigr) 
: |z| < R^\nu(\omega) \bigr\rbrace,
\]
where 
\(\Pi^0 \colon \prod_{j \geq 0} \mathbb{R}^m \to \mathbb{R}^m\) 
is the projection onto the first component, and
\[
\Gamma \colon 
F_{\mu_1}(\omega) 
\cap \bigl\lbrace z \in F_{\mu_1}(\omega) : \| z \| < R^\nu(\omega) \bigr\rbrace 
\to \prod_{j \geq 0} \mathbb{R}^m
\]
is defined in \cite[Lemma 2.7]{GVR23}. Moreover, \(\Gamma\) is a fixed point of the map \(I\), cf.~\cite[Lemma 2.6]{GVR23}, 
that is,
\[
I\!\bigl(v, \Gamma(v)\bigr) = \Gamma(v).
\]
Recall that \(F_{\mu_1}(\omega) = \mathbb{R}^m\). 
Also, by the final formula in \cite[page 122]{GVR23}, we have
\[
\Pi^0\!\bigl(I(z, \Gamma(z))\bigr) 
= \Pi^0\!\bigl(\Gamma(z)\bigr) 
= z.
\]
Consequently, from \cite[Theorem 2.10]{GVR23}, it follows that 
\begin{align}\label{AZA}
	S^{\nu}_{loc}(\omega) = \left\lbrace Y_\omega+z \, :\,  \vert z\vert< R^{\nu}(\omega)\right\rbrace.
\end{align}
This proves the first claim. If we choose \(0 < \nu_1 < \nu\), 
then by arguing as in Remark~\ref{cocnnn}, we obtain
\begin{align*}
	\Vert z\Vert< R^{\nu}(\omega)\longrightarrow \sup_{t\geq 0}\exp(t\nu_1)\Vert\varphi^{t}_{\omega}(z)-Y_{\theta_{t}\omega}\Vert<\infty.
\end{align*}
\end{proof}
We now aim to formulate sufficient conditions under which we can guarantee that 
the largest Lyapunov exponent \(\mu_1\) is strictly negative.
\begin{lemma}\label{lemma:upper_bound_top_lyapunov}
For the first Lyapunov exponent, we have
\[
\mu_1 
\leq \frac{1}{n t_0} 
\int_\Omega \log \| \psi^{n t_0}_\omega \| \, \mathbb{P}(\mathrm{d}\omega),
\quad \forall n \in \mathbb{N}.
\]

\end{lemma}
\begin{proof}
From \cite[3.3.2 Theorem]{Arn98}, it follows that
\begin{align*}
	\mu_1 = \inf_{n \geq 1} \frac{1}{nt_0} \int_{\Omega}\log \Vert \psi^{nt_0}_{\omega}\Vert\ \mathbb{P}(\mathrm{d}\omega) \leq  \frac{1}{nt_0}\int_{\Omega}\log \Vert \psi^{nt_0}_{\omega}\Vert\ \mathbb{P}(\mathrm{d}\omega).
\end{align*}
\end{proof}
\begin{example}\label{STABLE_EXAMPLE}
Let Assumption~\ref{S_U_C} hold, and consider the equation
\begin{align}\label{VVVA}
	\mathrm{d}Z_t 
	= V(Z_t) \, \mathrm{d} \mathbf{X}_t(\omega) 
	+ V_0(Z_t) \, \mathrm{d}t, 
	\quad Z_0 = z_0 \in \mathbb{R}^m,
\end{align}
with the additional assumption that \(V_0(0) = 0\) and \(V(0) = 0\).
In this case, \(Y_\omega \equiv 0\) is a stationary point. To see this, note that the statement holds for smooth paths \(\mathbf{X}(\omega)\). Since \(\mathbf{X}(\omega)\) is a geometric rough path, we can approximate it by smooth paths, and the statement remains true in the limit. Now assume that \(D_0 V_0\) has only eigenvalues with strictly negative real parts. Let
\[
\lambda 
= \max \bigl\lbrace 
\operatorname{Re}(\mu) 
: \mu \ \text{is an eigenvalue of } D_0 V_0 
\bigr\rbrace.
\]
Then, for every fixed \(n \in \mathbb{N}\), we have
\[
\frac{1}{nt_0}\int_{\Omega} \log \|\psi^{nt_0}_{\omega}\| \, \mathbb{P}(\mathrm{d}\omega) \to \frac{1}{nt_0} \log\left(\|\exp(nt_0 D_{0} V_0)\|_{\mathcal{L}(\mathbb{R}^m, \mathbb{R}^m)}\right),
\]
as \( \|D_{0} V\|_{\mathcal{L}(\mathbb{R}^m\otimes\mathbb{R}^d,\mathbb{R}^m)} \to 0 \). This can be proven using Proposition \ref{DDSD} and the Dominated Convergence Theorem. Furthermore, we have

\[
\lim_{n \rightarrow \infty} \frac{1}{nt_0} \log\left(\|\exp(nt_0 D_{0} V_0)\|_{\mathcal{L}(\mathbb{R}^m, \mathbb{R}^m)}\right) = \lambda < 0.
\]
Thus, we can choose \(n\) sufficiently large such that
\[
\frac{1}{nt_0} \log\left(\|\exp(nt_0 D_{0} V_0)\|_{\mathcal{L}(\mathbb{R}^m, \mathbb{R}^m)}\right) < 0.
\]
Hence, from Lemma~\ref{lemma:upper_bound_top_lyapunov}, it follows that
\[
\mu_1 < 0,
\]
as
\[
\|D_0 V\|_{\mathcal{L}(\mathbb{R}^m \otimes \mathbb{R}^d, \mathbb{R}^m)} \to 0.
\]
In other words, we have shown that when 
\(\|D_0 V\|_{\mathcal{L}(\mathbb{R}^m \otimes \mathbb{R}^d, \mathbb{R}^m)}\) 
is small, the first Lyapunov exponent is negative. Consequently, by Corollary~\ref{stability}, 
we obtain local exponential stability around the origin.
\end{example}
In the next example, we discuss the existence of invariant manifolds.			\begin{example}
Reconsider equation~\eqref{VVVA} under the assumptions \(V_0(0) = 0\) and \(V(0) = 0\).  
Note that the trajectory \(Y_\omega = 0\) is the stationary solution.  
Regardless of the specific form of \(D_0 V_0\), the multiplicative ergodic theorem guarantees the existence of Lyapunov exponents.
Consequently, we can use our results to obtain invariant manifolds according to the decomposition in \eqref{linn}.
A simple example where a center manifold exists is when  \( V(0) = D_{0}V = 0 \) and the matrix \( D_{0}V_0 \) has at least one eigenvalue on the unit circle. Indeed, when \(V(0) = D_0 V = 0\) and \(V_0(0) = 0\), the Lyapunov exponents near zero are entirely determined by \(D_0 V_0\). In this case, if the real part of one of the eigenvalues is strictly positive (respectively, strictly negative), we can deduce the existence of unstable (respectively, stable) manifolds.
\end{example}
\begin{remark}
An example where a nontrivial stationary point is expected to exist 
is when \(\mathbf{X}\) is a Brownian motion and \(V_0\) is a linear drift 
whose eigenvalues have strictly negative real parts. 
In this case, a nontrivial stationary point can be represented 
as the solution to the following equation:
\begin{align*}
	Y_{t}(\omega) = \int_{-\infty}^{t} \exp((t-s)V_0) V(Y_{s}(\omega))  \, \mathrm{d}\mathbf{B}_{s}(\omega).
\end{align*}
\end{remark}
\subsection*{Acknowledgements}
\label{sec:acknowledgements}
Mazyar Ghani Varzaneh acknowledges support from DFG CRC/TRR 388 {\em Rough Analysis, Stochastic Dynamics and Related Fields}, Project A06. 
\bibliographystyle{alpha}
\bibliography{refs}

\def\cprime{$'$} \def\cprime{$'$}
\begin{thebibliography}{FGGR16}

\bibitem[Arn98]{Arn98}
Ludwig Arnold.
\newblock {\em Random dynamical systems}.
\newblock Springer Monographs in Mathematics. Springer-Verlag, Berlin, 1998.

\bibitem[Box89]{Box89}
Petra Boxler.
\newblock A stochastic version of center manifold theory.
\newblock {\em Probab. Theory Related Fields}, 83(4):509--545, 1989.

\bibitem[BRS17]{BRS17}
Isma{\"e}l Bailleul, Sebastian Riedel, and Michael Scheutzow.
\newblock Random dynamical systems, rough paths and rough flows.
\newblock {\em J. Differential Equations}, 262(12):5792--5823, 2017.

\bibitem[CP11]{CP11}
Serge Cohen and Fabien Panloup.
\newblock Approximation of stationary solutions of {G}aussian driven stochastic
  differential equations.
\newblock {\em Stochastic Process. Appl.}, 121(12):2776--2801, 2011.

\bibitem[CPT14]{CPT14}
Serge Cohen, Fabien Panloup, and Samy Tindel.
\newblock Approximation of stationary solutions to {SDE}s driven by
  multiplicative fractional noise.
\newblock {\em Stochastic Process. Appl.}, 124(3):1197--1225, 2014.

\bibitem[CQ02]{CQ02}
Laure Coutin and Zhongmin Qian.
\newblock Stochastic analysis, rough path analysis and fractional {B}rownian
  motions.
\newblock {\em Probab. Theory Related Fields}, 122(1):108--140, 2002.

\bibitem[DH23]{DH23}
Luu~Hoang Duc and Phan~Thanh Hong.
\newblock Asymptotic dynamics of {Y}oung differential equations.
\newblock {\em J. Dynam. Differential Equations}, 35(2):1667--1692, 2023.

\bibitem[DHC19]{DHC19}
Luu~Hoang Duc, Phan~Thanh Hong, and Nguyen~Dinh Cong.
\newblock Asymptotic stability for stochastic dissipative systems with a
  {H}\"{o}lder noise.
\newblock {\em SIAM J. Control Optim.}, 57(4):3046--3071, 2019.

\bibitem[DPT19]{DPT19}
Aur\'{e}lien Deya, Fabien Panloup, and Samy Tindel.
\newblock Rate of convergence to equilibrium of fractional driven stochastic
  differential equations with rough multiplicative noise.
\newblock {\em Ann. Probab.}, 47(1):464--518, 2019.

\bibitem[Duc22]{Duc22}
Luu~Hoang Duc.
\newblock Exponential stability of stochastic systems: a pathwise approach.
\newblock {\em Stoch. Dyn.}, 22(3):Paper No. 2240012, 21, 2022.

\bibitem[FGGR16]{FGGR16}
Peter~K. Friz, Benjamin Gess, Archil Gulisashvili, and Sebastian Riedel.
\newblock The {J}ain-{M}onrad criterion for rough paths and applications to
  random {F}ourier series and non-{M}arkovian {H}\"ormander theory.
\newblock {\em Ann. Probab.}, 44(1):684--738, 2016.

\bibitem[FH20]{FH20}
Peter~K. Friz and Martin Hairer.
\newblock {\em A course on rough paths}.
\newblock Universitext. Springer, Cham, [2020] \copyright 2020.
\newblock With an introduction to regularity structures, Second edition of [
  3289027].

\bibitem[FP17]{FP17}
Joaquin Fontbona and Fabien Panloup.
\newblock Rate of convergence to equilibrium of fractional driven stochastic
  differential equations with some multiplicative noise.
\newblock {\em Ann. Inst. Henri Poincar\'{e} Probab. Stat.}, 53(2):503--538,
  2017.

\bibitem[FV10a]{FV10-2}
Peter~K. Friz and Nicolas~B. Victoir.
\newblock Differential equations driven by {G}aussian signals.
\newblock {\em Ann. Inst. Henri Poincar\'e Probab. Stat.}, 46(2):369--413,
  2010.

\bibitem[FV10b]{FV10}
Peter~K. Friz and Nicolas~B. Victoir.
\newblock {\em Multidimensional stochastic processes as rough paths}, volume
  120 of {\em Cambridge Studies in Advanced Mathematics}.
\newblock Cambridge University Press, Cambridge, 2010.
\newblock Theory and applications.

\bibitem[GANS18]{GANS18}
Mar\'{\i}a~J. Garrido-Atienza, Andreas Neuenkirch, and Bj\"{o}rn Schmalfu\ss.
\newblock Asymptotical stability of differential equations driven by
  {H}\"{o}lder continuous paths.
\newblock {\em J. Dynam. Differential Equations}, 30(1):359--377, 2018.

\bibitem[GAS18]{GAS18}
Mar\'{\i}a~J. Garrido-Atienza and Bj\"{o}rn Schmalfuss.
\newblock Local stability of differential equations driven by
  {H}\"{o}lder-continuous paths with {H}\"{o}lder index in {$(1/3,1/2)$}.
\newblock {\em SIAM J. Appl. Dyn. Syst.}, 17(3):2352--2380, 2018.

\bibitem[GVR21]{GVR21}
Mazyar Ghani~Varzaneh and Sebastian Riedel.
\newblock A dynamical theory for singular stochastic delay differential
  equations {II}: {N}onlinear equations and invariant manifolds.
\newblock {\em Discrete Contin. Dyn. Syst. Ser. B}, 26(8):4587--4612, 2021.

\bibitem[GVR23]{GVR23}
M.~Ghani~Varzaneh and S.~Riedel.
\newblock Oseledets {S}plitting and {I}nvariant {M}anifolds on {F}ields of
  {B}anach {S}paces.
\newblock {\em J. Dynam. Differential Equations}, 35(1):103--133, 2023.

\bibitem[GVR24]{GVR24}
Mazyar Ghani~Varzaneh and Sebastian Riedel.
\newblock Singular stochastic delay equations driven by fractional brownian
  motion: Dynamics, longtime behaviour, and pathwise stability.
\newblock {\em arXiv:2411.04590}, 2024.

\bibitem[GVR25a]{GVR25B}
Mazyar Ghani~Varzaneh and Sebastian Riedel.
\newblock A general center manifold theorem on fields of {B}anach spaces.
\newblock {\em Discrete Contin. Dyn. Syst. Ser. B}, 30(5):1499--1516, 2025.

\bibitem[GVR25b]{GVR25A}
Mazyar Ghani~Varzaneh and Sebastian Riedel.
\newblock An integrable bound for rough stochastic partial differential
  equations with applications to invariant manifolds and stability.
\newblock {\em Journal of Functional Analysis}, 288(1):110676, 2025.

\bibitem[Hai05]{Hai05}
Martin Hairer.
\newblock Ergodicity of stochastic differential equations driven by fractional
  {B}rownian motion.
\newblock {\em Ann. Probab.}, 33(2):703--758, 2005.

\bibitem[HO07]{HO07}
M.~Hairer and A.~Ohashi.
\newblock Ergodic theory for {SDE}s with extrinsic memory.
\newblock {\em Ann. Probab.}, 35(5):1950--1977, 2007.

\bibitem[HP11]{HP11}
M.~Hairer and N.~S. Pillai.
\newblock Ergodicity of hypoelliptic {SDE}s driven by fractional {B}rownian
  motion.
\newblock {\em Ann. Inst. Henri Poincar\'e Probab. Stat.}, 47(2):601--628,
  2011.

\bibitem[HP13]{HP13}
Martin Hairer and Natesh~S. Pillai.
\newblock Regularity of laws and ergodicity of hypoelliptic {SDE}s driven by
  rough paths.
\newblock {\em Ann. Probab.}, 41(4):2544--2598, 2013.

\bibitem[Kha12]{Kha12}
Rafail Khasminskii.
\newblock {\em Stochastic stability of differential equations}, volume~66 of
  {\em Stochastic Modelling and Applied Probability}.
\newblock Springer, Heidelberg, second edition, 2012.
\newblock With contributions by G. N. Milstein and M. B. Nevelson.

\bibitem[LCL07]{LCL07}
Terry~J. Lyons, Michael Caruana, and Thierry L{\'e}vy.
\newblock {\em Differential equations driven by rough paths}, volume 1908 of
  {\em Lecture Notes in Mathematics}.
\newblock Springer, Berlin, 2007.
\newblock Lectures from the 34th Summer School on Probability Theory held in
  Saint-Flour, July 6--24, 2004, With an introduction concerning the Summer
  School by Jean Picard.

\bibitem[MS99]{MS99}
Salah-Eldin~A. Mohammed and Michael Scheutzow.
\newblock The stable manifold theorem for stochastic differential equations.
\newblock {\em Ann. Probab.}, 27(2):615--652, 1999.

\bibitem[NK21]{NK21}
Alexandra Neam\c{t}u and Christian Kuehn.
\newblock Rough center manifolds.
\newblock {\em SIAM J. Math. Anal.}, 53(4):3912--3957, 2021.

\bibitem[PTV20]{PTV20}
Fabien Panloup, Samy Tindel, and Maylis Varvenne.
\newblock A general drift estimation procedure for stochastic differential
  equations with additive fractional noise.
\newblock {\em Electron. J. Stat.}, 14(1):1075--1136, 2020.

\bibitem[RS17]{RS17}
Sebastian Riedel and Michael Scheutzow.
\newblock Rough differential equations with unbounded drift term.
\newblock {\em J. Differential Equations}, 262(1):283--312, 2017.

\end{thebibliography}

\end{document}